\documentclass[journal]{IEEEtran}
\usepackage{cite}
\usepackage{amsmath}
\usepackage{amssymb}
\usepackage{amsthm}
\usepackage{arydshln}
\usepackage{mathtools}
\usepackage{mathrsfs}
\usepackage{algorithmic}
\usepackage{graphicx}
\usepackage{epsfig,graphics}
\usepackage{epstopdf}
\usepackage{subfigure}
\usepackage{enumitem}
\usepackage{pifont}

\usepackage{multirow,ragged2e}
\usepackage{bm}
\usepackage{fixltx2e}
\usepackage[english]{babel}

\usepackage{filecontents}
\theoremstyle{theorem}

\theoremstyle{lemma}
\newtheorem{lemma}{Lemma}
\theoremstyle{definition}
\newtheorem{definition}{Definition}
\theoremstyle{assumption}
\newtheorem{assumption}{Assumption}
\theoremstyle{problem}
\newtheorem{problem}{Problem}
\theoremstyle{example}

\theoremstyle{proposition}
\newtheorem{proposition}{Proposition}

\newtheorem{remark}{Remark}

\newcommand{\fig}[1]{Figure~\ref{#1}}
\newcommand{\figsa}[2]{Fig.~\ref{#1} and Fig.~\ref{#2}}

\newcommand{\sect}[1]{Section~\ref{#1}}
\newcommand{\sects}[2]{Sections~\ref{#1} and~\ref{#2}}
\newcommand{\eq}[1]{Equation~(\ref{#1})}
\newcommand{\eqs}[2]{Equations~(\ref{#1})-(\ref{#2})}
\newcommand{\ass}[1]{Assumption~\ref{#1}}
\newcommand{\pro}[1]{Proposition~\ref{#1}}
\newcommand{\lem}[1]{Lemma~\ref{#1}}
\newcommand{\rem}[1]{Remark~\ref{#1}}

\newcommand{\fbm}[1]{\mathbf{#1}}
\newcommand{\tbm}[1]{\fbm{#1}^\mathsf{T}}
\newcommand{\tfbm}[1]{\bm{#1}^\mathsf{T}}

\newcommand{\hbm}[1]{\hat{\fbm{#1}}}
\newcommand{\hfbm}[1]{\hat{\bm{#1}}}
\newcommand{\thbm}[1]{\hat{\fbm{#1}}^\mathsf{T}}
\newcommand{\tilbm}[1]{\tilde{\fbm{#1}}}
\newcommand{\tilfbm}[1]{\tilde{\bm{#1}}}

\newcommand{\ttilfbm}[1]{\tilde{\bm{#1}}^\mathsf{T}}
\newcommand{\dottbm}[1]{\dot{\fbm{#1}}^\mathsf{T}}
\newcommand{\dotbm}[1]{\dot{\fbm{#1}}}

\newcommand{\dothatbm}[1]{\dot{\hat{\fbm{#1}}}}
\newcommand{\dothatfbm}[1]{\dot{\hat{\bm{#1}}}}
\newcommand{\dothattbm}[1]{\dot{\hat{\fbm{#1}}}^\mathsf{T}}
\newcommand{\dottilbm}[1]{\dot{\tilde{\fbm{#1}}}}
\newcommand{\dottilfbm}[1]{\dot{\tilde{\bm{#1}}}}
\newcommand{\dotttilbm}[1]{\dot{\tilde{\fbm{#1}}}^\mathsf{T}}

\newcommand{\ddotbm}[1]{\ddot{\fbm{#1}}}

\newcommand{\ddothbm}[1]{\ddot{\hbm{#1}}}
\newcommand{\dddothbm}[1]{\dddot{\hbm{#1}}}

\newcommand{\Sat}[0]{\text{Sat}}
\newcommand{\sat}[0]{\text{sat}}
\newcommand{\atan}[0]{\text{atan}}
\newcommand{\sign}[0]{\text{sign}}

\begin{document}

\title{Connectivity-Preserving Consensus of Multi-Agent Systems with Bounded Actuation}

\author{Yuan~Yang,~\IEEEmembership{Student Member,~IEEE,} Daniela Constantinescu,~\IEEEmembership{Member,~IEEE,} and Yang Shi,~\IEEEmembership{Fellow,~IEEE}\thanks{The authors are with the Department of Mechanical Engineering, University of Victoria, Victoria, BC V8W 2Y2 Canada (e-mail: yangyuan@uvic.ca; danielac@uvic.ca; yshi@uvic.ca).}}

\maketitle

\begin{abstract}
This paper investigates the impact of bounded actuation on the connectivity-preserving consensus of two classes of multi-agent systems, with kinematic agents and with Euler-Lagrange agents. The investigation establishes that: (1) there exists a class of gradient-based controls which drive kinematic multi-agent systems to connectivity-preserving consensus even if they saturate; (2) actuator saturation restricts the initial states from which Euler-Lagrange multi-agent systems can be synchronized while preserving their local connectivity; (3) Euler-Lagrange multi-agent systems with unbounded actuation can achieve connectivity-preserving consensus without velocity measurements or exact system dynamics; and (4) a proposed indirect coupling control strategy drives Euler-Lagrange multi-agent systems with limited actuation and starting from rest to connectivity-preserving consensus without requiring velocity measurements and including in the presence of uncertain dynamics and time-varying delays.
\end{abstract}

\IEEEpeerreviewmaketitle

\section{Introduction}\label{sec:introduction}

Distributed coordination control of multi-agent systems~(MAS-s) drives all agents to the same state using only local and 1-hop state signals~\cite{Ren2005TAC}. Established strategies include Static Proportional~(P) control for first-order MAS-s~\cite{Francis2004TAC}, Proportional-Derivative~(PD) control for second-order MAS-s~\cite{Ren2008TAC}, and Proportional plus damping~(P+d) control for Euler-Lagrange networks~\cite{Nuno2013TRO}. Because practical inter-agent communications are distance-dependent, the connectivity assumption of conventional strategies may be violated during coordination~\cite{Zavlanos2011Proceedings}.

For kinematic MAS-s with first-order or nonholonomic agents, consensus can be formulated as the minimization of a potential energy function of inter-agent distances that has a unique minimum at the consensus configuration. A negative gradient-based controller can then drive the MAS to consensus. If the potential function is quadratic in inter-agent distances, the negative gradient law is a form of P control. Distributed P-type controls that guarantee the connectivity and coordination of kinematic MAS-s can be derived from unbounded~\cite{Egerstedt2007TRO,Dimos2007TAC,Dimos2008TRO,Zavlanos2008TRO} or bounded~\cite{Amir2010TAC,Dimos2008ICRA,Dimos2010IET,Amir2013TAC,Wen2012IET,Dixon2015TCNS} potentials. Non-smooth gradient-based controls can guarantee finite-time consensus in the presence of disturbances~\cite{Dong2016Automatica} and Lipschitz nonlinearities~\cite{Ren2016TAC}. Other distributed gradient-based strategies provide connectivity in the presence of actuator saturation~\cite{Khorasani2013ACC} or of obstacles~\cite{Dixon2012TAC, Dimos2017TAC}, strong connectivity in directed graphs in the presence of disturbances~\cite{Spong2017TAC}, or intermittent connectivity~\cite{Hollinger2012TRO,Banfi2016ICRA,Zavlanos2017TAC}. Recent research investigates the robustness and invariance of connectivity preservation in the presence of additional control terms~\cite{Dimos2017SIAM}, and the trade-offs among bounded controls, connectivity maintenance and additional control objectives~\cite{Sabattini2017TRO}. Nonetheless, the effect of actuation bounds on the connectivity-preserving consensus of kinematic MAS-s is incompletely elucidated. This paper will identify a class of gradient-based controls that drive kinematic MAS-s to connectivity-preserving consensus even if saturated.

For MAS-s with second-order, including Euler-Lagrange, agents, consensus can be formulated as the minimization of an energy function with unique minimum at the consensus state and with two components: a potential energy function like that of kinematic MAS-s; and a kinetic energy function of agent velocities with unique minimum at the consensus velocity. A negative gradient plus damping injection strategy, like conventional PD and P+d control, can then drive the second-order MAS to consensus. In the absence of actuation constraints, distributed PD controls can also guarantee connectivity-preserving consensus for double-integrator MAS-s~\cite{Su2010SCL}. Robust gradient-based laws can maintain connectivity during the coordination to consensus of double-integrators with Lipschitz-like dynamic nonlinearities~\cite{Su2011Automatica}, and during leader-follower coordination of double-integrators~\cite{Dong2013Automatica,Dong2014TAC,Su2015Automatica,Hu2015IJRNC,Ai2016Automatica} and of Euler-Lagrange agents~\cite{Ren2012SCL,Dong2017IJRNC}. Integral terms added to sliding mode and conventional PD controllers can robustly preserve connectivity during rendezvouz~\cite{Hu2017TCNS}, flocking~\cite{Dong2015Automatica} and formation tracking~\cite{Hu2017ACC}. Decentralized algebraic connectivity estimation can preserve global connectivity in cooperative control of multi-robots~\cite{Sabattini2013IJRR,Paolo2013IJRR,Sabattini2013TRO,Sabattini2015Cybernetics,Sabattini2015AJC}. A question still open is whether actuation bounds thwart the connectivity-preserving consensus of second-order MAS-s. This paper will show that the answer depends on the initial state of the MAS. 

At the communications level, recurrent proximity maintenance~\cite{Hollinger2012TRO,Banfi2016ICRA,Zavlanos2017TAC}, switching graphs~\cite{Williams2013ICRA,Williams2013ICRA2,Williams2013TRO}, directed graphs~\cite{Spong2017TAC} and intermittent algebraic connectivity estimators~\cite{Williams2015ICRA,Williams2017TRO} tackle threats to connectivity due to limited agent communication ranges. Threats due to time-varying communication delays are considered only for attitude synchronization~\cite{Tayebi2012TAC}, for Euler-Lagrange MAS-s with uncertain parameters~\cite{Nuno2011TAC} and for Euler-Lagrange MAS-s without velocity measurements~\cite{Nuno2018TCST}. The dangers posed to the connectivity-preserving consensus of Euler-Lagrange MAS-s by combined communication delays and limited actuation are unclear. This paper will show how to overcome those combined dangers for Euler-Lagrange MAS-s that start from rest even if they have uncertain dynamics and only position measurements.

The paper contributes to research on connectivity-preserving consensus of MAS-s with bounded actuation as follows:
\begin{itemize}
\item First, by regarding saturated actuation as scaling of the planned controls, the paper proves that there exists a class of gradient-based strategies which can drive kinematic MAS-s to connectivity-preserving consensus even if the actuators saturate. These strategies are practically important because their design is unconstrained by the actuator design/selection. Unconstrained controller design can also improve system performance~\cite{Teel2004TRA}. Simulations in~\sect{sec: simulations} verify that unconstrained gradient-based control is simpler to design and can drive a kinematic MAS to connectivity-preserving consensus faster than the saturation-dependent strategy in~\cite{Khorasani2013ACC}.
\item Second, the paper shows that second-order MAS-s with bounded actuation cannot achieve connectivity-preserving consensus from some initial states. This conclusion arises from an intrinsic conflict between connectivity preservation and limited actuation in second-order MAS-s, illustrated through an exemplary $2$-agent system.
\item Third, the paper develops an output feedback controller and an adaptive controller to drive fully actuated Euler-Lagrange MAS-s to connectivity-preserving consensus using only position measurements and uncertain dynamics, respectively. The two controllers show that connectivity can be preserved from any initial state if the actuators do not saturate, by selecting either the coupling stiffness or the injected damping suitably large. They also indicate the need for a methodology to guarantee connectivity-preserving consensus subject to actuation bounds.
\item 
Fourth, the paper develops an indirect coupling framework that overcomes the conflict between bounded actuation and connectivity maintenance in Euler-Lagrange MAS-s which start from rest. The framework: introduces dynamic proxies for each agent; connects communicating agents through their proxies; treats the agent-proxy couplings subject to actuator saturation as in single robot regulation~\cite{Rio2007TRO}, where control gains can be freely tuned; and converts the actuation bounds into bounds on the stiffness of the inter-proxy couplings by minimizing the potential energy of agent-proxy couplings on the boundary of a ball. To the best knowledge of the authors, the indirect coupling framework in this paper is the first to handle connectivity preservation, time-varying delays, limited actuation and system uncertainties or lack of velocity measurements, simultaneously. Its key benefits are that: (1) it can constrain proxies tightly, because the virtual controls can be arbitrarily large; and (2) it enables free tuning of the agent-proxy couplings and, thus, better use of the bounded actuation. 
\end{itemize}

\section{Connectivity Preservation}\label{sec: connectivity preservation}

This section first presents the definitions and properties of, and the assumptions on, the MAS communications needed in the following sections. Then, it introduces a class of potential functions that generalizes prior potentials widely-used in connectivity-preserving consensus.

Consider a MAS with similar agents $i$, each with position $\fbm{q}_{i}\in\mathbb{R}^{n}$ and with limited communication capability $r$. Two agents $i$ and $j$ are adjacent, or neighbours, if and only if they (1) can, and (2) agree to, exchange information with each other. The two agents cannot exchange information and be adjacent if the distance between them is larger than, or equal to, their communication distance, i.e., $\|\fbm{q}_{ij}\|=\|\fbm{q}_{j}-\fbm{q}_{i}\| = \|\fbm{q}_{ji}\|=\|\fbm{q}_{i}-\fbm{q}_{j}\|\ge r$.
The agents can communicate when closer to each other than their communication distance, i.e.,  $\|\fbm{q}_{ij}\|=\|\fbm{q}_{ji}\|< r$, 
but they need to also agree to exchange information to be neighbours. Thus, the sets of agent neighbours need not change when coordination is achieved and all agents are within communication distance of all other agents.

The following definition of, and assumption on, the communication graph of the MAS are adopted in this paper.
\begin{definition}\label{def1}\cite{Egerstedt2010Princeton}
The communication graph of a MAS $\mathcal{G}=\{\mathcal{V},\mathcal{E}\}$ consists of a set of nodes $\mathcal{V}=\{1,\cdots,N\}$, each associated with one agent in the system, and a set of communication edges $\mathcal{E}=\{(i,j)\in \mathcal{V}\times\mathcal{V}|i\in\mathcal{N}_{j}\}$, each associated with a communication link in the system.
\end{definition}
\begin{assumption}\label{ass1}
The initial communication graph $\mathcal{G}(0)$ is undirected, i.e., $(i,j)\in\mathcal{E}(0)$ if and only if $(j,i)\in\mathcal{E}(0)$.
\end{assumption}

It follows from the above definition of neighbouring agents that the four statements below are equivalent:
\begin{enumerate}
\item[1.]
Agents $i$ and $j$ are neighbours of each other;
\item[2.]
Agents $i$ and $j$ belong the set of neighbours of each other, i.e., $i\in\mathcal{N}_{j}$ and $j\in\mathcal{N}_{i}$;
\item[3.]
The communication links~(edges of the MAS communication graph) $(i,j)$ and $(j,i)$ exist;
\item[4.]
Agents $i$ and $j$ are within communication distance of and agree to exchange information with each other.
\end{enumerate}

A path in the graph $\mathcal{G}$ is a sequence of connected edges $(i,j)$, $(j,k)$, $\cdots$. Further, the graph $\mathcal{G}$ is connected if and only if there exists a path between each pair of agents. The associated weighted adjacency matrix $\fbm{A}=[a_{ij}]$ of an undirected communication graph $\mathcal{G}$ is symmetric, with $a_{ij}>0$ if $(j,i)\in\mathcal{E}$, and $a_{ij}=0$ otherwise. Correspondingly, the weighted Laplacian matrix $\fbm{L}=[l_{ij}]$ of $\mathcal{G}$ is symmetric with
\begin{align*}
l_{ij}=\begin{cases}
\sum_{k\in\mathcal{N}_{i}}a_{ik}\quad &j=i\\
-a_{ij}\quad &j\neq i
\end{cases}\textrm{.}
\end{align*}
Let the undirected communication graph $\mathcal{G}$ contain $2M$ edges. Label only one of the edges $(i,j)$ and $(j,i)$ as $e_{k}$, $k=1,\cdots,M$, with weight $w(e_{k})=a_{ji}=a_{ij}$. For example, $e_{k}=(j,i)$ means agents $j$ and $i$ are the tail and the head of $e_{k}$, respectively. Then, the incidence matrix $\fbm{D}=[d_{hk}]$ of the graph $\mathcal{G}$ is defined by
\begin{align*}
d_{hk}=\begin{cases}
1\quad &\text{if agent}\ h\ \text{is the head of } e_{k}\textrm{,}\\
-1\quad&\text{if agent}\ h\ \text{is the tail of } e_{k}\textrm{.}\\
0 \quad&\text{otherwise}
\end{cases}
\end{align*}
It encodes edge orientation (from tail to head) and is related to the Laplacian $\fbm{L}$ of $\mathcal{G}$ through:
\begin{lemma}\label{lem1}\cite{Egerstedt2010Princeton}
Given an arbitrary orientation of the edge set $\mathcal{E}$, the weighted Laplacian matrix of the undirected communication graph $\mathcal{G}$ can be decomposed as
\begin{align*}
\fbm{L}=\fbm{D}\fbm{W}\tbm{D}\textrm{,}
\end{align*}
where $\fbm{W}$ is a $M\times M$ diagonal matrix with $w(e_{k})$, $k=1,\cdots,M$, on the diagonal. 
\end{lemma}

Another assumption used in this paper is:
\begin{assumption}\label{ass2}
The undirected communication graph of the MAS is initially connected and each pair of initially adjacent agents $(i,j)$ is strictly within their communication distance, i.e., $\|\fbm{q}_{ij}(0)\|=\|\fbm{q}_{ji}(0)\|\leq r-\epsilon$ for some $\epsilon>0$.
\end{assumption}

\begin{remark}
\normalfont\ass{ass2} is widely used in gradient-based connectivity-preserving consensus control~\cite{Egerstedt2007TRO,Dimos2007TAC,Dimos2008TRO,Zavlanos2008TRO,Amir2010TAC,Dimos2008ICRA,Dimos2010IET,Amir2013TAC,Wen2012IET,Dixon2015TCNS} to ensure that the potential function providing the controller is strictly smaller than its maximum intially. \normalfont\ass{ass2} is equivalent to assumption $\|\fbm{q}_{ij}(0)\|<r$ in~\cite{Amir2010TAC,Dimos2008ICRA,Dimos2010IET,Amir2013TAC,Wen2012IET,Dixon2015TCNS} if the controller design does not require $\epsilon$, like for example in kinematic MAS strategies based on the generalized potential in~\sect{sec: first-order}. However, the potential in~\cite{Amir2010TAC,Dimos2008ICRA,Dimos2010IET,Amir2013TAC,Wen2012IET,Dixon2015TCNS} requires sophisticated parameter selections, as shown in~\cite{Dixon2012TAC}. In contrast, the generalized potential in~\sect{sec: first-order}, whether bounded or unbounded, leads to simpler design and is applicable to kinematic MAS-s with both full and limited actuation. The mediated coupling strategy for Euler-Lagrange MAS-s in~\sect{sec: EL} uses $\epsilon$ to decompose actuation constraints and preserve connectivity.
\end{remark}

The objective of this paper is to drive MAS-s with kinematic and Euler-Lagrange agents with bounded actuation to consensus while preserving their initial connectivity. Connectivity preservation, i.e., $\mathcal{E}(t)=\mathcal{E}(0)$ $\forall t\ge 0$, requires all edges of the initial communication graph to be maintained, i.e., $(i,j)\in\mathcal{E}(0) \Rightarrow (i,j)\in\mathcal{E}(t)$ $\forall t\ge 0$. For MAS-s whose agents have the same communication radius $r$, connectivity preservation becomes the problem of keeping $d_{ij}(t)\leq r$ for all $t\geq 0$ and for all $(i,j)\in\mathcal{E}(0)$. Consensus can be formulated as the minimization of a potential function by driving the MAS with the corresponding negative gradient-based control law, as often done in existing work. Then, connectivity-preserving consensus can be formulated as the problem of bounding and simultaneously minimizing a suitable potential function. The remainder of this section will show that different potential functions which are widely adopted in existing connectivity-preserving coordination research are particular forms of a generalized potential. It will also prove that suitably bounding the generalized potential is equivalent to maintaining the initial connectivity and that minimizing it is equivalent to driving the MAS to consensus.

Consider a MAS with $n$ agents and the generalized potential function 
\begin{equation}\label{equ1}
\Psi(\fbm{q})=\frac{1}{2}\sum^{n}_{i=1}\sum_{j\in\mathcal{N}_{i}(0)}\psi(\|\fbm{q}_{ij}\|)\textrm{,}
\end{equation}
where $\fbm{q}=[\tbm{q}_{1},\cdots,\tbm{q}_{n}]^\mathsf{T}$, with $\fbm{q}_{i}$ the position of agent $i$; $\mathcal{N}_{i}(0)$ is the set of neighbours of agent $i$ at $t=0$; and $\psi(\|\fbm{q}_{ij}\|)$ and $\Psi(\fbm{q})$ obey:
\begin{enumerate}
\item[1.]
$\psi(\|\fbm{q}_{ij}\|)\ge 0$ and $\|\fbm{q}_{ij}\|=0 \Leftrightarrow \psi(\|\fbm{q}_{ij}\|)=0$;
\item[2.]
$\frac{\partial\psi(\|\fbm{q}_{ij}\|)}{\partial\|\fbm{q}_{ij}\|^{2}}$ is positive-definite and exists at every point $\|\fbm{q}_{ij}\|\in [0,r)$;
\item[3.]
$\Psi(\fbm{q})=\Psi_{max}$ when there exists $(i,j)\in\mathcal{E}(0)$ such that $\|\fbm{q}_{ij}(t)\|=r$, where $\Psi_{max}$ is not necessarily bounded.
\end{enumerate}

Note that the proposed potential function~\eqref{equ1} generalizes some widely-used potential functions, including the unbounded functions~\cite{Egerstedt2007TRO,Dimos2007TAC,Dimos2008TRO,Zavlanos2008TRO} and a typical bounded function~\cite{Su2010SCL,Khorasani2013ACC}. Consider, for example, the unbounded functions: 
\begin{equation}\label{equ2}
\psi^{u}(\|\fbm{q}_{ij}\|)=\frac{\|\fbm{q}_{ij}\|^{2}}{r^{2}-\|\fbm{q}_{ij}\|^{2}}
\end{equation}
and the bounded functions:
\begin{equation}\label{equ3}
\psi^{b}(\|\fbm{q}_{ij}\|)=\frac{\|\fbm{q}_{ij}\|^{2}}{r^{2}-\|\fbm{q}_{ij}\|^{2}+Q}
\end{equation}
with $Q>0$ to be determined. Both $\psi^{u}(\|\fbm{q}_{ij}\|)$ and $\psi^{b}(\|\fbm{q}_{ij}\|)$ are positive definite and increasing with respect to $\|\fbm{q}_{ij}\|$ on $[0,r)$, and are zero if $\|\fbm{q}_{ij}\|=0$. Thus, they satisfy the first two properties of~\eqref{equ1}. By \ass{ass2}, $0\leq\psi^{u}(\|\fbm{q}_{ij}(0)\|)<\infty$ and $0\leq\psi^{b}(\|\fbm{q}_{ij}(0)\|)<\frac{r^{2}}{Q}$. Suppose that $\|\fbm{q}_{ij}(t^{-})\|\in [0,r)$, i.e., $0\leq\psi^{u}(\|\fbm{q}_{ij}(t^{-})\|)<\infty$ and $0\leq\psi^{b}(\|\fbm{q}_{ij}(t^{-1})\|)<\frac{r^{2}}{Q}$. Then, $\|\fbm{q}_{ij}(t^{+})\|\in [0,r)$ if $0\leq\psi^{u}(\|\fbm{q}_{ij}(t^{+})\|)<\infty$, or if $0\leq\psi^{b}(\|\fbm{q}_{ij}(t^{+})\|)<\frac{r^{2}}{Q}$. Let the maximum of $\Psi$ be $\Psi_{max}=\infty$ for $\psi^{u}(\|\fbm{q}_{ij}\|)$ and $\Psi_{max}=\frac{r^{2}}{Q}$ for $\psi^{b}(\|\fbm{q}_{ij}\|)$, respectively. By continuity of $\Psi(\fbm{q})$ on $[0,r)$, $\Psi(\fbm{q})<\Psi_{max}$ is sufficient for $\|\fbm{q}_{ij}\|<r$, i.e., for $\psi^{u}(\|\fbm{q}_{ij}\|)<\infty$ or $\psi^{b}(\|\fbm{q}_{ij}\|)<\frac{r^{2}}{Q}$ for each edge $(i,j)\in\mathcal{E}(0)$, which corresponds to the third property of~\eqref{equ1}. Hence, the prior unbounded~\cite{Egerstedt2007TRO,Dimos2007TAC,Dimos2008TRO,Zavlanos2008TRO} and bounded~\cite{Su2010SCL,Khorasani2013ACC} potential functions are particular versions of the general potential function \eqref{equ1}.

Another function used for connectivity-preserving consensus in~\cite{Amir2010TAC,Dimos2008ICRA,Dimos2010IET,Amir2013TAC,Wen2012IET,Dixon2015TCNS} is the navigation function:
\begin{equation}\label{equ4}
\Psi_{i}=\frac{\gamma_{i}}{(\gamma^{k}_{i}+\beta_{i})^{\frac{1}{k}}}\textrm{,}
\end{equation}
for agent $i$, with $k>0$ a constant and
\begin{align*}
\gamma_{i}=\sum_{j\in\mathcal{N}_{i}(0)}\frac{1}{2}\|\fbm{q}_{ij}\|^{2}\ \text{and}\ \beta_{i}=\prod_{j\in\mathcal{N}_{i}(0)}\frac{1}{2}(r^{2}-\|\fbm{q}_{ij}\|^{2})\textrm{.}
\end{align*}
Each $\Psi_{i}$ is positive definite, zero if $\|\fbm{q}_{ij}\|=0$ $\forall j\in\mathcal{N}_{i}(0)$, increases with respect to $\|\fbm{q}_{ij}\|$ on $[0,r)$ and achieves its maximum $\Psi_{max}=1$ when $\beta_{i}=0$, i.e., when there exists $(i,j)\in\mathcal{E}(0)$ such that $\|\fbm{q}_{ij}\|=r$. Hence, $\Psi_{i}(t)<\Psi_{max} \Rightarrow \|\fbm{q}_{ij}(t)\|<r$ $\forall (i,j)\in\mathcal{E}(0)$, and agent control laws based on $\Psi_{i}$ instead of \eqref{equ1} are used to maintain the initial connectivity. However, the analysis of the impact of velocity saturation on controls based on~\eqref{equ1} presented in~\sect{sec: first-order} also applies to controls based on $\Psi_{i}$.

The following proposition states the use of the generalized potential function~\eqref{equ1} in connectivity-preserving consensus:
\begin{proposition}\label{pro1}
Given an initially connected MAS with undirected communication edges $\mathcal{E}(0)$ and with potential function $\Psi(0)<\Psi_{max}$ defined in \eqref{equ1}: (i) all initial communication edges are preserved, i.e., $(i,j)\in\mathcal{E}(0)\Rightarrow (i,j)\in\mathcal{E}(t)$ $\forall t\geq 0$, if and only if $0\leq \Psi(t)<\Psi_{max}$ $\forall t\geq 0$; and (ii) all agents converge to the same configuration, i.e., $\fbm{q}_{i}\to\fbm{c}$, if and only if $\Psi(t)\to 0$ as $t\to \infty$.
\end{proposition} 
\begin{proof}
\item[(i)]
Let $(i,j)\in\mathcal{E}(t)$ $\forall (i,j)\in\mathcal{E}(0)$. Then, by the first and third properties of~\eqref{equ1}, all $\psi(\|\fbm{q}_{ij}(t)\|)\geq 0$ and $\Psi(t)<\Psi_{max}$. Hence, $0\leq\Psi(t)\le \Psi_{max}$ $\forall t\ge 0$.

\noindent Conversely, let $0\leq\Psi(t)<\Psi_{max}$, $\forall t\geq 0$, and assume there exists $(i,j)\in\mathcal{E}(0)$ such that $(i,j)\notin\mathcal{E}(t_{2})$ at time instant $t_{2}$, i.e., $\|\fbm{q}_{ij}\|\geq r$. Because $\fbm{q}_{ij}$ is continuous in $t$, there exists $0<t_{1}\leq t_{2}$ such that $\|\fbm{q}_{ij}(t_{1})\|=r$. Then, by the third property of \eqref{equ1}, $\Psi(t_{1})=\Psi_{max}$, which contradicts $\Psi(t)<\Psi_{max}$ for any $t\geq 0$.
\item[(ii)]
Let all agents converge to the same configuration, i.e., $\fbm{q}_{i}\to\fbm{c}$ $\forall i$. Then, $\forall (i,j)\in\mathcal{E}(0)$, $\|\fbm{q}_{ij}\|\to 0$ and $\psi(\|\fbm{q}_{ij}\|)\to 0$ by the first property of \eqref{equ1}, and further $\Psi(t)\to 0$ as $t\to \infty$.\\
Conversely, let $\Psi(t)\to 0$ as $t\to \infty$ and assume $\exists (i,j)\in\mathcal{E}(0)$ such that $\fbm{q}_{i}-\fbm{q}_{j}\nrightarrow\fbm{0}$. By the first property of \eqref{equ1}, it follows that $\exists (i,j)\in \mathcal{E}(0)$ such that $\psi(\|\fbm{q}_{ij}\|)\nrightarrow 0$ and, from~\eqref{equ1}, that $\Psi(t)\nrightarrow 0$, which contradicts $\Psi(t)\to 0$ as $t\to \infty$.\\
\end{proof}

\pro{pro1} shows that connectivity-preserving consensus can be formulated as the following consensus problem with inter-agent distance constraints:
\begin{problem}
Given a MAS with bounded actuation and satisfying \ass{ass2}, find a distributed control law such that $d_{ij}(t)<r$ $\forall t\geq 0$ and $d_{ij}(t)\to 0$ as $t\to \infty$ $\forall (i,j)\in\mathcal{E}(0)$.
\end{problem}

\begin{remark}
\normalfont In general, $d_{ij}(t)\to 0$ as $t\to \infty$ does not imply that $d_{ij}(t)<r$ $\forall t\geq 0$, i.e., consensus is a more general problem than connectivity-preserving consensus. 
\end{remark} 

\section{Kinematic Networks}\label{sec: first-order}

This section shows that kinematic MAS-s can be driven to connectivity-preserving consensus with gradient-based controls derived from the potential function~\eqref{equ1} whether actuator saturation prevents their full application to the system or not.

\subsection{Single-Integrator Systems}\label{sec: single-integrator}
Consider a MAS with $N$ single-integrator agents:
\begin{equation}\label{equ5}
\dotbm{q}_{i}=\fbm{u}_{i}\textrm{,}
\end{equation}
where $i=1,\cdots,N$ indexes the agents; and $\fbm{q}_{i}=\begin{pmatrix}q_{i1},\cdots,q_{in}\end{pmatrix}^\mathsf{T}\in\mathbb{R}^{n}$ and $\fbm{u}_{i}=\begin{pmatrix}u_{i1},\cdots,u_{in}\end{pmatrix}^\mathsf{T}\in\mathbb{R}^{n}$ are the position and the the actual control of agent $i$, respectively.

The potential function~\eqref{equ1} provides the following control:
\begin{equation}\label{equ6}
\hbm{u}_{i}=-\sum_{j\in\mathcal{N}_{i}(0)}\nabla_{i}\psi(\|\fbm{q}_{ij}\|)\textrm{,}
\end{equation}
where $\nabla_{i}\psi(\|\fbm{q}_{ij}\|)$ is the gradient of $\psi(\|\fbm{q}_{ij}\|)$ with respect to $\fbm{q}_{i}$. By the second property of~\eqref{equ1}, $\psi(\|\fbm{q}_{ij}\|)$ is differentiable with respect to $\|\fbm{q}_{ij}\|^{2}$ and
\begin{equation}\label{equ7}
\begin{aligned}
\nabla_{i}\psi(\|\fbm{q}_{ij}\|)=\frac{2\partial\psi(\|\fbm{q}_{ij}\|)}{\partial\|\fbm{q}_{ij}\|^{2}}(\fbm{q}_{i}-\fbm{q}_{j})\textrm{,}
\end{aligned}
\end{equation}
with $\frac{\partial\psi(\|\fbm{q}_{ij}\|)}{\partial\|\fbm{q}_{ij}\|^{2}}$ is positive definite. Hence, the control~\eqref{equ6} is a type of nonlinear P control with state-dependent gains.

The actual actuation $\fbm{u}_{i}$ applied to agent $i$ is equal to the designed control $\hbm{u}_{i}$ if $\hbm{u}_{i}$ is within the agent's actuation bound; otherwise, the actuator saturates and applies only part of $\hbm{u}_{i}$:
\begin{equation}\label{equ8}
\fbm{u}_{i}=\Sat_{i}(\hbm{u}_{i})=\begin{pmatrix}\sat_{i1}(\hat{u}_{i1}) & \cdots &\sat_{in}(\hat{u}_{in})\end{pmatrix}^\mathsf{T}\textrm{,}
\end{equation}
where $\hat{u}_{ik}$ is the $k$-th element of $\hbm{u}_{i}$ and $\sat_{ik}(\cdot)$ is the standard saturation function, a widely used model for actuator saturation~\cite{Rio2007TRO}.

This paper regards the hard physical constraints imposed on the designed control by the saturation of the actuators as automatic scaling through time-varying positive gains:
\begin{equation}\label{equ9}
\fbm{u}_{i}=\Sat_{i}(\hbm{u}_{i})=\fbm{S}_{i}(t)\hbm{u}_{i}\textrm{,}
\end{equation}
where $\fbm{S}_{i}(t)=\text{diag}\{s_{i1}(t),\cdots,s_{in}(t)\}, i=1,\cdots,N$ are diagonal matrices with time-varying diagonal elements $s_{ik}(t)>0$ for all $k=1,\cdots,n$. This dynamic scaling model of saturation $s_{ik}(t)$  is related to the standard model of saturation $\sat(\hat{u}_{ik})$. The standard model is $\sat(\hat{u})=\hat{u}$ if $\underline{u}\leq\hat{u}\leq\overline{u}$; $\sat(\hat{u})=\overline{u}$ if $\hat{u}>\overline{u}$; and $\sat(\hat{u})=\underline{u}$ otherwise. An equivalent dynamic scaling model is $\sat(\hat{u})=s(t)\hat{u}$, where: $s(t)=1$ if $\underline{u}\leq\hat{u}\leq\overline{u}$; $s(t)=\frac{\overline{u}}{\hat{u}}$ if $\hat{u}>\overline{u}$; and $s(t)=\frac{\underline{u}}{\hat{u}}$ otherwise. Thus, the actuator saturation model in this paper is similar to~\cite{Rio2007TRO}. Hence, the current design inherits the merits of the design~\cite{Rio2007TRO}.

\begin{remark}
\normalfont \eq{equ9} is key to how this paper accounts for actuation bounds in connectivity-preserving consensus of kinematic MAS-s. The dynamic scaling matrices $\fbm{S}_{i}(t)$ are not an artificial construction but a mathematical model of saturation, and the design in this paper places no constraints on the control other than the specified actuation bounds, unlike the control in~\cite{Khorasani2013ACC}. Therefore, it verifies the claim in~\cite{Teel2004TRA,Rio2007TRO} that control designs free of artificial constraints improve performance, as shown through simulations in~\sect{sec: simulations}. 
\end{remark}

By \pro{pro1}, the connectivity-preserving coordination of \eqref{equ5} under the control~\eqref{equ9} can be evaluated using $V(t)=\Psi(\fbm{q})$. To this end, the derivative of $V(t)$ can be computed by:
\begin{align*}
&\dot{V}(t)=\frac{1}{2}\sum^{N}_{i=1}\sum_{j\in\mathcal{N}_{i}(0)}\dot{\psi}(\|\fbm{q}_{ij}\|)\\
=&\frac{1}{2}\sum^{N}_{i=1}\sum_{j\in\mathcal{N}_{i}(0)}\dottbm{q}_{i}\nabla_{i}\psi(\|\fbm{q}_{ij}\|)+\frac{1}{2}\sum^{N}_{i=1}\sum_{j\in\mathcal{N}_{i}(0)}\dottbm{q}_{j}\nabla_{j}\psi(\|\fbm{q}_{ij}\|)\\
=&\frac{1}{2}\sum^{N}_{i=1}\sum_{j\in\mathcal{N}_{i}(0)}\dottbm{q}_{i}\nabla_{i}\psi(\|\fbm{q}_{ij}\|)+\frac{1}{2}\sum^{N}_{j=1}\sum_{i\in\mathcal{N}_{j}(0)}\dottbm{q}_{i}\nabla_{i}\psi(\|\fbm{q}_{ij}\|)\\
=&\frac{1}{2}\sum^{N}_{i=1}\sum_{j\in\mathcal{N}_{i}(0)}\dottbm{q}_{i}\nabla_{i}\psi(\|\fbm{q}_{ij}\|)+\frac{1}{2}\sum^{N}_{i=1}\sum_{j\in\mathcal{N}_{i}(0)}\dottbm{q}_{i}\nabla_{i}\psi(\|\fbm{q}_{ij}\|)\\
=&\sum^{N}_{i=1}\sum_{j\in\mathcal{N}_{i}(0)}\dottbm{q}_{i}\nabla_{i}\psi(\|\fbm{q}_{ij}\|)\textrm{,}
\end{align*}
where the third equality follows from $\psi(\|\fbm{q}_{ij}\|)=\psi(\|\fbm{q}_{ji}\|)$, and the fourth from~\ass{ass1}. Using \eqref{equ5}, \eqref{equ6} and \eqref{equ9}, $\dot{V}(t)$ can be rearranged into:
\begin{equation}\label{equ10}
\begin{aligned}
\dot{V}=-\sum^{N}_{i=1}\thbm{u}_{i}\fbm{S}_{i}(t)\hbm{u}_{i}\textrm{.}
\end{aligned}
\end{equation}
Because $\fbm{S}_{i}(t)$ are positive definite, $\dot{V}(t)\leq 0$ and $V(t)\leq V(0)$, $\forall t\geq 0$. By \ass{ass2} and \pro{pro1}, $V(0)<\Psi_{max}$. Hence, $\Psi(\fbm{q})\leq V(0)<\Psi_{max}$ and, by the third property of~\eqref{equ1}, all initial communication links are maintained.

By the first two properties of~\eqref{equ1} and the Lasalle Invariance Principle, the MAS trajectories tend to the largest invariant set in $\mathcal{I}=\{\fbm{q}:\ \dot{V}=0\}$. The above proven connectivity preservation implies that $\|\fbm{q}_{ij}\|<r$ for all $j\in\mathcal{N}_{i}(0)$ and $t\ge 0$. From \eq{equ7}, it follows that $\nabla_{i}\psi(\|\fbm{x}_{ij}\|)$ are bounded, and thus that $\hbm{u}_{i}$ are bounded. Bounded $\hbm{u}_{i}$ imply that the dynamic scaling factors $s_{ik}(t)$ are lower bounded by some positive constants. Then, \eq{equ10} leads to $\mathcal{I}=\{\fbm{q}:\ \sum_{j\in\mathcal{N}_{i}(0)}\nabla_{i}\psi(\|\fbm{q}_{ij}\|)=\fbm{0}\textrm{,}\ i=1,\cdots,N\}$. \eq{equ7} together with $\fbm{c}_{k}=\begin{pmatrix}q_{1k} & \cdots & q_{Nk}\end{pmatrix}^\mathsf{T}$, $k=1,\cdots,n$ permit to rewrite $\sum_{j\in\mathcal{N}_{i}}\nabla_{i}\psi(\|\fbm{q}_{ij}\|)\to\fbm{0}$ for $i=1,\cdots,N$ as $\fbm{L}(\fbm{q})\fbm{c}_{k}\to\fbm{0}$ for all $k=1,\cdots,n$, where $\fbm{L}(\fbm{q})=[l_{ij}(\fbm{q})]$ with
\begin{align*}
l_{ij}(\fbm{q})=\begin{cases}
-\sum_{k\in\mathcal{N}_{i}(0)}l_{ik}(\fbm{q})\quad &j=i\textrm{,}\\
-\frac{2\partial\psi(\|\fbm{q}_{ij}\|)}{\partial\|\fbm{q}_{ij}\|^{2}}\quad &j\neq i\ \text{and}\ j\in\mathcal{N}_{i}(0)\textrm{,}\\
0\quad &j\neq i\ \text{and}\ j\notin\mathcal{N}_{i}(0)\textrm{.}
\end{cases}
\end{align*}
The state-dependent $\fbm{L}(\fbm{q})$ is the weighted Laplacian of the undirected communication graph of the single-integrator MAS, which, given Lemma~\ref{lem1} and the guaranteed preservation of initial communication links, can be decomposed as $\fbm{L}(\fbm{q})=\fbm{D}\fbm{W}(\fbm{q})\tbm{D}$, where the diagonal matrix  $\fbm{W}(\fbm{q})$ has $w(e_{k})=a_{ij}=\frac{2\partial\psi(\|\fbm{q}_{ij}\|)}{\partial\|\fbm{q}_{ij}\|^{2}}$ on its diagonal. Further, $\fbm{L}(\fbm{q})\fbm{c}_{k}\to\fbm{0}$ leads to $\tbm{c}_{k}\fbm{L}(\fbm{q})\fbm{c}_{k}=\left(\tbm{D}\fbm{c}_{k}\right)^\mathsf{T}\fbm{W}(\fbm{q})\left(\tbm{D}\fbm{c}_{k}\right)\to 0$, $k=1,\cdots,n$. Because $\fbm{W}(\fbm{q})$ is positive definite, it follows that $\tbm{D}\fbm{c}_{k}\to\fbm{0}$ and that $q_{ik}-q_{jk}\to 0$ for each pair of neighbouring agents $(i,j)\in\mathcal{E}(0)$. Then, the guaranteed preservation of the initially connected communication graph yields that $q_{1k}\to q_{2k}\to\cdots\to q_{Nk}$ for $k=1,\cdots,n$, or $\fbm{q}_{1}\to\cdots\to\fbm{q}_{N}$ and coordination is achieved. 

\begin{remark}
\normalfont The proof above leads to an important conclusion: velocity saturation is no threat to the connectivity-preserving consensus of single-integrator MAS-s under conventional negative gradient-based control derived from the generalized potential function~\eqref{equ1}. Critical to the proof is the dynamic scaling model of actuator saturation in~\eqref{equ9} because it enables the quadratic decomposition in~\eqref{equ10}. Connectivity preservation makes $\fbm{S}_{i}(t)$ positive definite, and the Lasalle Invariance Principle together with the decomposition~\eqref{equ10} and the positive definiteness of $\fbm{S}_{i}(t)$ guarantee the coordination of the first-order MAS. The practical significance of the conclusion is that controllers based on the generalized potential function~\eqref{equ1} are unconstrained by the actuator design/selection and need no modifications to account for actuator saturation. Controllers unconstrained by actuator saturation exploit the system actuation better and converge faster, as observed for single robot control in~\cite{Teel2004TRA,Rio2007TRO} and verified for single-integrator MAS-s through simulation comparison to the controller~\cite{Khorasani2013ACC} in~\sect{sec: simulations}. 
\end{remark}

\subsection{Nonholonomic Systems}\label{sec: nonholonomic} 

Let a MAS have $N$ nonholonomic agents with dynamics:
\begin{equation}\label{equ11}
\begin{aligned}
\dot{x}_{i}=&v_{i}\cos(\theta_{i})\\
\dot{y}_{i}=&v_{i}\sin(\theta_{i}) \qquad i=1,\cdots,N\\
\dot{\theta}_{i}=&\omega_{i}
\end{aligned}
\end{equation}
where $\fbm{q}_{i}=\begin{pmatrix}x_{i}, y_{i}\end{pmatrix}^\mathsf{T}\in\mathbb{R}^{2}$ and $\theta_{i}\in (-\pi, \pi]$ are the position and orientation of agent $i$ in a global coordinate frame. 

The negative gradient-based control derived from the generalized potential function~\eqref{equ1} for the nonholonomic MAS is:
\begin{equation}\label{equ12}
\begin{aligned}
\hat{v}_{i}=&\|\hat{\fbm{u}}_{i}\|\cos(\tilde{\theta}_{i})\\
\hat{\omega}_{i}=&-k\tilde{\theta}_{i}
\end{aligned}
\end{equation}
where: $\hat{\fbm{u}}_{i}=\begin{pmatrix}\hat{u}_{ix}, \hat{u}_{iy}\end{pmatrix}^\mathsf{T}$ is defined in~\eqref{equ6}; $\tilde{\theta}_{i}=\theta_{i}-\hat{\theta}_{i}$ with $\hat{\theta}_{i}=\atan2 (\hat{u}_{iy}, \hat{u}_{ix})$; and $k$ is any constant positive gain.

Due to actuator saturation, the controls actually applied to agent $i$ and the controls designed in~\eqref{equ12} are related through:
\begin{equation}\label{equ13}
\begin{aligned}
v_{i}=&\sat_{vi}(\hat{v}_{i})=s_{vi}(t)\hat{v}_{i}\textrm{,}\\
\omega_{i}=&\sat_{\omega i}(\hat{\omega}_{i})=s_{\omega i}(t)\hat{\omega}_{i}\textrm{,}
\end{aligned}
\end{equation}
where $\sat_{vi}(\cdot)$ and $\sat_{\omega i}(\cdot)$ are standard saturation functions, and $s_{vi}(t)>0$ and $s_{\omega}(t)>0$ are time-varying gains which this paper regards as dynamic scalings of the designed controls.

\begin{remark}
\normalfont The static P controller for the agent orientation in~\eqref{equ12} indicates the alignment between the orientation and the designed linear velocity $\hat{v}_{i}$ of the agent.

In~\eqref{equ12}, the orientation control of each agent $i$ is a simple P control, which means that each agent's orientation is aligned with its designed linear velocity $\hat{v}_{i}$. The position control of agent $i$ is the projection of the potential function's negative gradient on its orientation $\theta_{i}$. This guarantees that each agent moves in a direction decreasing the potential function. Since the dynamic scaling factors $s_{vi}(t)$ and $s_{\omega i}(t)$ are positive, actual orientation and position actuations of agent $i$ continue to align its orientation in the direction of $\hat{v}_{i}$ and to decrease the potential function.
\end{remark}


From \eqs{equ11}{equ13}, the linear velocity of agent $i$ is:
\begin{equation}\label{equ14}
\begin{aligned}
\dotbm{q}_{i}=&\begin{bmatrix}\dot{x}_{i}\\ \dot{y}_{i}\end{bmatrix}=v_{i}\begin{bmatrix}\cos(\theta_{i})\\ \sin(\theta_{i})\end{bmatrix}=s_{vi}(t)\hat{v}_{i}\begin{bmatrix}\cos(\theta_{i})\\ \sin(\theta_{i})\end{bmatrix}\\
=&s_{vi}(t)\|\hat{\fbm{u}}_{i}\|\cos(\tilde{\theta}_{i})\begin{bmatrix}\cos(\theta_{i})\\ \sin(\theta_{i})\end{bmatrix}\\
=&\frac{1}{2}s_{vi}(t)\|\hat{\fbm{u}}_{i}\|\begin{bmatrix}\cos(2\theta_{i}-\hat{\theta}_{i})+\cos(\hat{\theta}_{i})\\ \sin(2\theta_{i}-\hat{\theta}_{i})+\sin(\hat{\theta}_{i})\end{bmatrix}\\
=&\frac{1}{2}s_{vi}(t)\|\hat{\fbm{u}}_{i}\|\left(\text{Rot}(2\tilde{\theta}_{i})+\fbm{I}_{2}\right)\begin{bmatrix}\cos(\hat{\theta}_{i})\\ \sin(\hat{\theta}_{i})\end{bmatrix}\textrm{,}
\end{aligned}
\end{equation}
where:
\begin{align*}
\text{Rot}(2\tilde{\theta}_{i})=\begin{bmatrix}\cos(2\tilde{\theta}_{i})&-\sin(2\tilde{\theta}_{i})\\ \sin(2\tilde{\theta}_{i})&\cos(2\tilde{\theta}_{i})\end{bmatrix}\textrm{.}
\end{align*}

Connectivity preservation in the nonholonomic MAS with the bounded control~\eqref{equ13} is investigated using the same $V(t)=\Psi(\fbm{q})$ as in the single-integrator MAS, with derivative:
\begin{equation}\label{equ15}
\begin{aligned}
&\dot{V}(t)=\sum^{N}_{i=1}\sum_{j\in\mathcal{N}_{i}(0)}\dottbm{q}_{i}\nabla_{i}\psi(\|\fbm{q}_{ij}\|)=-\sum^{N}_{i=1}\thbm{u}_{i}\dotbm{q}_{i}\\
&=-\frac{1}{2}\sum^{N}_{i=1}s_{vi}(t)\|\hat{\fbm{u}}_{i}\|^{2}\begin{bmatrix}\cos(\hat{\theta}_{i})\\ \sin(\hat{\theta}_{i})\end{bmatrix}^\mathsf{T}(\text{Rot}(2\tilde{\theta}_{i})+\fbm{I}_{2})
\begin{bmatrix}\cos(\hat{\theta}_{i})\\ \sin(\hat{\theta}_{i})\end{bmatrix}\\
&=-\frac{1}{2}\sum^{N}_{i=1}s_{vi}(t)\|\hat{\fbm{u}}_{i}\|^{2}\left(1+\cos(2\tilde{\theta}_{i})\right)\leq 0\textrm{.}
\end{aligned} 
\end{equation}
\eq{equ15} indicates that $V$ decreases monotonically along the system trajectories, in particular $V(t)\leq V(0)$ $\forall t\ge 0$. \ass{ass2} and \pro{pro1} lead to $V(0)<\Psi_{max}$ and further $\Psi(\fbm{q})\leq V(0)<\Psi_{max}$. Thus, the controller~\eqref{equ12} preserves the initial connectivity of the nonholonomic MAS. 

The first two properties of $\Psi(\fbm{q})$ and the Lasalle Invariance Principle imply that the system trajectories converge to the largest invariant set in $\mathcal{I}=\{\fbm{p}: \dot{V}=0\}$, where $\fbm{p}=\begin{bmatrix}\tbm{p}_{1}, \cdots, \tbm{p}_{N}\end{bmatrix}^\mathsf{T}$ and $\fbm{p}_{i}=\begin{bmatrix}\tbm{q}_{i}, \theta_{i}\end{bmatrix}^\mathsf{T}$, $i=1,\cdots,N$. Preservation of initial connectivity, i.e.,  $\|\fbm{q}_{ij}\|<r$ for all $(i,j)\in \mathcal{E}(0)$, and~\eqref{equ6} imply that $\hat{\fbm{u}}_{i}$ are bounded, and further that $s_{vi}(t)$ are lower-bounded by some positive constants for $i=1,\cdots,N$. Then, $\dot{V}(t)\leq 0$ in~\eqref{equ15} leads to $\mathcal{I}=\left\{\fbm{p}: \|\hat{\fbm{u}}_{i}\|^{2}\left(1+\cos(2\tilde{\theta}_{i})\right)=0,\ i=1,\cdots,N\right\}$. 

A proof by contradiction can now show that convergence to the invariant set $\mathcal{I}$ implies $\lim\limits_{t\to\infty}\hbm{u}_{i}=0$. Convergence to a set in $\mathcal{I}$ together with \eqref{equ12} imply that $\hat{v}_{i}=\|\hat{\fbm{u}}_{i}\|\cos(\tilde{\theta}_{i})\to 0$, and thus $\dotbm{q}_{i}=s_{vi}(t)\hat{v}_{i}\begin{bmatrix}\cos(\theta)\ \sin(\theta)\end{bmatrix}^\mathsf{T}\to\fbm{0}$. Now assume that $\tilde{\theta}_{i}\to\pm\frac{\pi}{2}$ and $\|\hat{\fbm{u}}_{i}\|\nrightarrow 0$. Then, the derivative of $\hat{\theta}_{i}$ is:
\begin{equation}\label{equ16}
\begin{aligned}
\dot{\theta}^{\ast}_{i}=&\frac{\dothattbm{u}_{i}}{\|\hat{\fbm{u}}_{i}\|}\begin{bmatrix}-\sin(\hat{\theta}_{i})\\ \cos(\hat{\theta}_{i})\end{bmatrix}=-\frac{1}{\|\hat{\fbm{u}}_{i}\|}\sum_{j\in\mathcal{N}_{i}}\Big[\nabla^{2}_{i}\psi(\|\fbm{q}_{ij}\|)\dotbm{q}_{i}\\
&\ \ \ +\nabla_{j}\nabla_{i}\psi(\|\fbm{q}_{ij}\|)\dotbm{q}_{j}\Big]^\mathsf{T}\begin{bmatrix}-\sin(\hat{\theta}_{i})\\ \cos(\hat{\theta}_{i})\end{bmatrix}\to 0\textrm{.}
\end{aligned}
\end{equation}
and $\hat{\omega}_{i}=-k\tilde{\theta}_{i}\to\mp\frac{k\pi}{2}$, so 
\begin{equation}\label{equ17}
\dot{\theta}_{i}=\omega_{i}=\sat_{\omega i}(\hat{\omega}_{i})\to-\text{sgn}(\tilde{\theta}_{i})\min\{\overline{s}_{\omega i}, k|\tilde{\theta}_{i}|\}\neq 0\textrm{.}
\end{equation} 
\eqs{equ16}{equ17} imply that $\dot{\tilde{\theta}}_{i}=\dot{\theta}_{i}-\dot{\theta}^{\ast}_{i}\nrightarrow 0$, which contradicts the assumption that $\tilde{\theta}_{i}\to\pm\frac{\pi}{2}$. Therefore, the largest invariant set to which the trajectories of the nonholonomic MAS converge must be contained in $\mathcal{I}_{u}=\{\fbm{p}: \|\hat{\fbm{u}}_{i}\|=0, i=1,\cdots,N\}$, and thus $\lim\limits_{t\to\infty}\hat{\fbm{u}}_{i}=-\lim\limits_{t\to\infty}\sum_{j\in\mathcal{N}_{i}(0)}\nabla_{i}\psi(\|\fbm{q}_{ij}\|)=\fbm{0}$ for $i=1,\cdots,N$. Lastly, the guaranteed connectivity of the communication graph together with an analysis similar to the one in~\sect{sec: single-integrator} lead to the conclusion that $\fbm{q}_{1}\to\cdots\to\fbm{q}_{N}$, i.e, that position coordination is guaranteed.

\begin{remark}
\normalfont The proof above leads to a similar conclusion as for single-integrator MAS-s: actuator saturation does not threaten the connectivity-preserving consensus of nonholonomic MAS-s under conventional negative gradient-based control derived from the generalized potential function~\eqref{equ1}. Key to the proof is that the factors $s_{vi}(t)$, which scale the designed control in~\eqref{equ15}, are positively lower-bounded and, thus, maintain the monotonic decrease of $V(t)$. As discussed in \sect{sec: single-integrator}, these scaling factors are not introduced artificially, through control design. They are an exact model of the actuator saturation intrinsic to the MAS. Hence, the connectivity-preserving control law is designed in~\eqref{equ12}, unconstrained by actuator saturation, and the actual velocities are~\eqref{equ13} due to limited actuation. Simulations in \sect{sec: simulations} verify that the P control aligns the orientation of all agents with their designed linear velocity regardless of velocity saturation, similar to~\cite{Rio2007TRO}.
\end{remark}

\subsection{Dynamic Graphs}\label{sec: dynamic}

\sects{sec: single-integrator}{sec: nonholonomic} have proven that the gradient-based control laws \eqref{equ6} and \eqref{equ12}, even when saturated due to limited actuations~(\eqref{equ8} and~\eqref{equ13}), guarantee both connectivity preservation for, and coordination of, single-integrator and nonholonomic MAS, respectively. Guaranteed coordination implies that all agents move within communication distance of all other agents after some time, for any initial connectivity of the MAS. Consider a $5$-agent MAS, whose communication graph is a cycle $C_{5}$ at $t=0$. At some time $t>0$, each agent becomes sufficiently close to all other agents for the communications graph of the MAS to potentially become a complete graph $K_{5}$.
It then seems reasonable that the MAS could be coordinated faster if agents would become adjacent to new agents as they move within communication distance of each other, i.e., if new communication links $\mathcal{E}(K_{5})-\mathcal{E}(C_{5})$ would be established and the sensing graph would become dynamic. This section modifies the gradient-based controllers~\eqref{equ6} and \eqref{equ12} to exploit such additional communication while preserving the initial connectivity and coordinating the kinematic MAS.

The redesign involves two steps. The first step accounts for a dynamic sensing graph $\mathcal{G}(t)$ by replacing the static set $\mathcal{N}_{i}(0)$ of neighbours of agent $i$ with a dynamic set $\mathcal{N}_{i}(t)$ of neighbours of agent $i$. The second step preserves the initial connectivity through the following hysteresis mechanism: two agents who were not adjacent at $t=0$ become neighbours only at a time $t$ when they are within $r-\epsilon$ distance of each other, i.e., if $j\notin\mathcal{N}_{i}(t^{-})$ and $\|\fbm{x}_{ij}(t)\|< r-\epsilon$, then $j\in\mathcal{N}_{i}(t)$.

Now let $t_1$ be the time instant when the first edge is added to the communication graph. The same assumptions and similar analysis as in \sects{sec: single-integrator}{sec: nonholonomic} lead to the conclusion that $V=\Psi(\fbm{q})$ decreases monotonically during $[0,t_{1}]$, $V(t_{1})\leq V(0)$. Therefore, no existing graph edge is broken at $t_1$. Given the finite number of agents $N$ in the MAS, it follows by induction that $V(t_{k})\leq V(0), \forall t\in[0, t_{k}]$ for all $k=1, \cdots, k_{max}$, where $k_{max}$ is the maximum number of edges added to the initial communication graph $\mathcal{G}(0)$. Because the existing edges are preserved when a new edge is established, $\mathcal{G}(t)$ becomes a complete graph at $t_{k_{max}}$. Thereafter, the proofs in \sects{sec: single-integrator}{sec: nonholonomic} guarantee that the proposed controllers preserve the connectivity and coordinate the MAS whether they saturate or not.

\section{Networked Euler-Lagrange Systems}\label{sec: EL}
This section starts by recalling the properties of Euler-Lagrange dynamics. Afterwards, an exemplary $2$-agent system exposes the intrinsic conflict between connectivity preservation and bounded actuation in second-order MAS-s, and demonstrates that actuator saturation prevents these systems from achieving connectivity-preserving consensus from arbitrary initial state. Then, the section proves that full actuation is sufficient to drive an Euler-Lagrange MAS to connectivity-preserving consensus from any initial state using only position measurements and uncertain dynamics. Lastly, the section develops an indirect coupling control framework based on the generalized potential function~\eqref{equ1}. This framework yields controllers that drive Euler-Lagrange MAS-s which start from rest to connectivity-preserving consensus even if the MAS-s have bounded actuation, system uncertainties, only position measurements and time-varying communication delays.

Let a MAS have $N$ non-redundant Euler-Lagrange agents with dynamics:
\begin{equation}\label{equ18}
\fbm{M}_{i}(\fbm{q}_{i})\ddotbm{q}_{i}+\fbm{C}_{i}(\fbm{q}_{i},\dotbm{q}_{i})\dotbm{q}_{i}+\fbm{g}_{i}(\fbm{q}_{i})=\fbm{f}_{i}\textrm{.}
\end{equation}
In~\eq{equ18}: the subscript $i=1,\cdots,N$ indexes the agent; $\fbm{q}_{i}$, $\dotbm{q}_{i}$ and $\ddotbm{q}_{i}$ are its position, velocity and acceleration; $\fbm{M}_{i}(\fbm{q}_{i})$ and $\fbm{C}_{i}(\fbm{q}_{i},\dotbm{q}_{i})$ are its matrices of inertia and of Coriolis and centrifugal effects, respectively; $\fbm{g}_{i}(\fbm{q}_{i})$ is its force of gravity; and $\fbm{f}_{i}$ is the control force applied to the agent.  

The dynamics in~\eqref{equ18} have the following properties~\cite{Kelly2006Springer}:
\begin{enumerate}[label=P.\arabic*]
\item \label{P1}
The inertia matrix $\fbm{M}_{i}(\fbm{q}_{i})$ is symmetric, positive definite and uniformly bounded by $\fbm{0}\prec\lambda_{i1}\fbm{I}\preceq \fbm{M}_{i}(\fbm{q}_{i})\preceq \lambda_{i2}\fbm{I}\prec\infty$, with $\lambda_{i1}>0, \lambda_{i2}>0$.
\item \label{P2}
The matrix $\dot{\fbm{M}}_{i}(\fbm{q}_{i})-2\fbm{C}_{i}(\fbm{q}_{i},\dotbm{q}_{i})$ is skew-symmetric.
\item \label{P3}
There exists $c_{i}>0$ such that $\|\fbm{C}_{i}(\fbm{q}_{i},\dotbm{q}_{i})\fbm{y}\|\leq c_{i}\|\dotbm{q}_{i}\|\|\fbm{y}\|$, $\forall \fbm{q}_{i}\textrm{,} \dotbm{q}_{i}\textrm{,} \fbm{y}$.
\item \label{P4}
The dynamics~\eqref{equ18} admit a linear parameterization of the form: $\fbm{M}_{i}(\fbm{q}_{i})\ddotbm{x}_{ri}+\fbm{C}_{i}(\fbm{q}_{i},\dotbm{q}_{i})\dotbm{x}_{ri}+\fbm{g}_{i}(\fbm{q}_{i})=\bm{\Phi}_{i}(\fbm{q}_{i},\dotbm{q}_{i},\dotbm{x}_{ri},\ddotbm{x}_{ri})\bm{\theta}_{i}$, where $\bm{\Phi}_{i}(\fbm{q}_{i},\dotbm{q}_{i},\dotbm{x}_{ri},\ddotbm{x}_{ri})$ is a regressor matrix of known functions and $\underline{\bm{\theta}}_{i}\leq\bm{\theta}_{i}\leq\overline{\bm{\theta}}_{i}$ is a constant vector containing system parameters. 
\end{enumerate} 
For simplicity of notation, matrix and vector dependencies on $\fbm{q}_{i}$ and $\dotbm{q}_{i}$ are omitted in the remainder of this paper, for example, $\fbm{C}_{i}$ and $\fbm{g}_{i}$ indicate $\fbm{C}_{i}(\fbm{q}_{i},\dotbm{q}_{i})$ and $\fbm{g}_{i}(\fbm{q}_{i})$, respectively.

\subsection{Bounded Actuation - Connectivity Preservation Conflict}\label{sec: conflict}

This section uses a proof by contradiction to show that second-order MAS-s with limited actuation cannot achieve connectivity-preserving consensus from certain initial state, i.e., to show that Problem~\ref{pro1} is generally infeasible for second-order MAS-s.

To this end, assume that Problem~\ref{pro1} is feasible, i.e., that any second-order MAS with bounded actuation can be driven to consensus from any initial state while preserving its initial connectivity. Let the $2$-agent MAS have: dynamics $\ddot{x}_{i}=f_{i}$, $i=1,2$, communication radius $r>0$, and maximum actuations $\bar{f}_1$ and $\bar{f}_2$; and initial state such that the agents are a distance $d_{12}(0)=x_{2}(0)-x_{1}(0)=r-\epsilon$ apart for some $0<\epsilon<r$, and move away from each other with relative velocity $\dot{x}_{2}(0)-\dot{x}_{1}(0)=\sqrt{3(r-d_{12}(0))(\bar{f}_{1}+\bar{f}_{2})}$. Let the MAS have any controller. To maintain connectivity, the controller should stop the increase of $d_{12}(t)$, i.e., it should drive the relative velocity $\dot{x}_{2}(t)-\dot{x}_{1}(t)$ to zero, while $d_{12}(t)<r$. Due to the limited actuation, the fastest rate at which any controller can decrease $\dot{x}_{2}(t)-\dot{x}_{1}(t)$ is $-\bar{f}_{1}-\bar{f}_{2}$. Direct calculation lead to $\dot{x}_{2}(t)-\dot{x}_{1}(t)=\sqrt{(r-d_{12}(0))(\bar{f}_{1}+\bar{f}_{2})}>0$ at the time instant when $d_{12}(t)=r$. Therefore, no controller can stop the two agents from moving apart while they are in communication distance of each other. In other words, no controller can maintain the connectivity of the $2$-agent MAS and drive it to consensus. This is a contradiction with the hypothesis. Hence, Problem~\ref{pro1} is generally infeasible for second-order MAS-s.

\subsection{Velocity Estimation}\label{sec: output}

This section shows that a gradient-based control law derived from either an unbounded or a bounded generalized potential function~\eqref{equ1} can drive a fully actuated Euler-Lagrange MAS-s with only position sensing to connectivity-preserving consensus from any initial state. 

Let such a MAS have a first-order filter to estimate velocities and the following output feedback coordinating controller:
\begin{equation}\label{equ19}
\begin{aligned}
\fbm{f}_{i}=&-c_{i}\sum_{j\in\mathcal{N}_{i}(0)}\nabla_{i}\psi(\|\fbm{q}_{ij}\|)-\kappa_{i}\dothatbm{q}_{i}+\fbm{g}_{i}\textrm{,}\\
\dothatbm{q}_{i}=&-\hbm{q}_{i}+\kappa_{i}\fbm{q}_{i}\textrm{,}
\end{aligned}
\end{equation}
where: $\fbm{f}_{i}$ and $\dothatbm{q}_{i}$ are the designed control force and the estimated velocity of agent $i$, respectively; and $c_{i}$ and $\kappa_{i}$ are positive constants.

The following Lyapunov candidate function serves the study of the Euler-Lagrange MAS~\eqref{equ18} under the control~\eqref{equ19}:
\begin{equation}\label{equ20}
V=\frac{1}{2}\sum^{N}_{i=1}\left[\frac{1}{c_{i}}\dottbm{q}_{i}\fbm{M}_{i}\dotbm{q}_{i}+\frac{1}{c_{i}}\dothattbm{q}_{i}\dothatbm{q}_{i}+\sum_{j\in\mathcal{N}_{i}(0)}\psi(\|\fbm{q}_{ij}\|)\right] \textrm{.}
\end{equation}
The derivative of $V$ can be computed using property~\ref{P2} and the derivative of the filter dynamics: 
\begin{equation}\label{equ21}
\begin{aligned}
\dot{V}=&\frac{1}{2}\sum^{N}_{i=1}\sum_{j\in\mathcal{N}_{i}(0)}\dot{\psi}(\|\fbm{q}_{ij}\|)-\sum^{N}_{i=1}\dottbm{q}_{i}\sum_{j\in\mathcal{N}_{i}(0)}\nabla_{i}\psi(\|\fbm{q}_{ij}\|)\\
&-\sum^{N}_{i=1}\frac{\kappa_{i}}{c_{i}}\dottbm{q}_{i}\dothatbm{q}_{i}-\frac{1}{c_{i}}\sum^{N}_{i=1}\dothattbm{q}_{i}\dothatbm{q}_{i}+\sum^{N}_{i=1}\frac{\kappa_{i}}{c_{i}}\dothatbm{q}_{i}\dotbm{q}_{i}\textrm{.}
\end{aligned}
\end{equation}
From \ass{ass1} and $\psi(\|\fbm{q}_{ij}\|)=\psi(\|\fbm{q}_{ji}\|)$, it follows that:
\begin{align*}
\frac{1}{2}\sum^{N}_{i=1}\sum_{j\in\mathcal{N}_{i}(0)}\dot{\psi}(\|\fbm{q}_{ij}\|)=\sum^{N}_{i=1}\dottbm{q}_{i}\sum_{j\in\mathcal{N}_{i}(0)}\nabla_{i}\psi(\|\fbm{q}_{ij}\|),
\end{align*}
which, together with~\eqref{equ21}, leads to:
\begin{equation}\label{equ22}
\dot{V}=-\frac{1}{c_{i}}\sum^{N}_{i=1}\dothattbm{q}_{i}\dothatbm{q}_{i}\leq 0\textrm{.}
\end{equation}
From $\dottbm{q}_{i}\fbm{M}_{i}\dotbm{q}_{i}\geq 0$ and $\dothattbm{q}_{i}\dothatbm{q}_{i}\geq 0$, it follows that $V\ge \Psi(\fbm{q})$, where $\Psi(\fbm{q})$ is the generalized potential function in~\eqref{equ1}. If $V(0)\leq \Psi_{max}$, then~\eqref{equ22} implies that $\Psi(t)\leq V(t)\leq V(0)\leq\Psi_{max} $ for any $t\geq 0$ and, by the third property of $\Psi(\fbm{q})$, that the initial MAS connectivity is maintained.

The above analysis shows that a $V$ that obeys $V(0)<\Psi_{max}$ is sufficient for connectivity preservation. \ass{ass2} guarantees $\Psi(\fbm{q}(0))<\Psi_{max}$. Because $\dothatbm{q}_{i}$ are designed dynamics, $\dothatbm{q}_{i}(0)$ and, with them, the second sum in $V(0)$ can be set zero by choosing $\hbm{q}_{i}(0)=\kappa_{i}\fbm{q}_{i}(0)$. The first sum in $V(0)$, however, depends on the initial kinetic energy of the MAS. Therefore, for fixed $c_{i}$, the condition $V(0)<\Psi_{max}$ is infeasible for arbitrarily large initial velocities if $\Psi_{max}$ is finite, and is automatically satisfied if $\Psi_{max}$ is infinite. If $c_{i}$ and, with them, the controls $\fbm{f}_{i}$ can be arbitrarily large, the first sum in $V(0)$ can be reduced arbitrarily and $V(0)<\Psi_{max}$ can be guaranteed for any bounded $\Psi(\fbm{q})$ and any initial velocities. Thus, either selecting $\Psi(\fbm{q})$ unbounded or increasing the gradient-based controls~\eqref{equ19} derived from a bounded $\Psi(\fbm{q})$ guarantees $V(0)<\Psi_{max}$ and preserves the initial connectivity of a fully actuated Euler-Lagrange MAS with only position measurements. 

\begin{remark}
\normalfont The first and the third properties of the generalized potential~\eqref{equ1} lead to $\psi(\|\fbm{q}_{ij}\|)\to \Psi_{max}=\infty$ as $\|\fbm{q}_{ij}\|\to r$ if $\Psi(\fbm{q})$ is unbounded. Then, the second property of~\eqref{equ1} implies that the partial derivatives $\frac{\partial\psi(\|\fbm{q}_{ij}\|)}{\partial\|\fbm{q}_{ij}\|^{2}}$ are unbounded.
Unbounded $\frac{\partial\psi(\|\fbm{q}_{ij}\|)}{\partial\|\fbm{q}_{ij}\|^{2}}$ imply continuously increasing and unbounded controls as $\|\fbm{q}_{ij}\|\to r$. Hence, both an unbounded $\Psi(\fbm{q})$ and suitably enlarged gradients of a bounded $\Psi(\fbm{q})$ in~\eqref{equ19} act similarly: they increase the controls continuously to sufficiently large attractive forces when initial communication links $(i,j)\in\mathcal{E}(0)$ are threatened.  
\end{remark}

The analysis of coordination uses Barbalat's lemma and follows the conventional synchronization analysis of Euler-Lagrange networks. It is omitted here to save space.

\subsection{System Uncertainties}\label{sec: adaptive}

For fully actuated Euler-Lagrange MAS-s with uncertain parameters, this section overcomes the inability to compensate gravity terms directly by designing the following linear filter-based adaptive control law:
\begin{equation}\label{equ23}
\begin{aligned}
\fbm{f}_{i}=&\bm{\Phi}_{i}(\fbm{q}_{i},\dotbm{q}_{i},\fbm{e}_{i},\dotbm{e}_{i})\hfbm{\theta}_{i}-\mu_{i}\fbm{s}_{i}-\kappa_{i}\dotbm{q}_{i}\textrm{,}\\
\dothatfbm{\theta}_{i}=&-\tfbm{\Phi}_{i}(\fbm{q}_{i},\dotbm{q}_{i},\fbm{e}_{i},\dotbm{e}_{i})\fbm{s}_{i}\textrm{,}
\end{aligned}
\end{equation}
where: $i=1,\cdots,N$; $\mu_{i}$ and $\kappa_{i}$ are positive constants; $\fbm{s}_{i}=\dotbm{q}_{i}+\alpha\fbm{e}_{i}$ with $\alpha>0$ and $\fbm{e}_{i}=\sum_{j\in\mathcal{N}_{i}(0)}\nabla_{i}\psi(\|\fbm{q}_{ij}\|)$; and $\bm{\Phi}_{i}(\fbm{q}_{i},\dotbm{q}_{i},\fbm{e}_{i},\dotbm{e}_{i})\hfbm{\theta}_{i}=\hbm{M}_{i}(\fbm{q}_{i})(-\alpha\dotbm{e}_{i})+\hbm{C}_{i}(\fbm{q}_{i},\dotbm{q}_{i})(-\alpha\fbm{e}_{i})+\hbm{g}_{i}(\fbm{q}_{i})$. 

Adding $-\bm{\Phi}_{i}(\fbm{q}_{i},\dotbm{q}_{i},\fbm{e}_{i},\dotbm{e}_{i})\bm{\theta}_{i}$ on both sides of~\eqref{equ18} and using~\eqref{equ23} and property~\ref{P4} lead to the closed-loop dynamics:
\begin{equation}\label{equ24}
\begin{aligned}
\fbm{M}_{i}(\fbm{q}_{i})\dotbm{s}_{i}+\fbm{C}_{i}(\fbm{q}_{i},\dotbm{q}_{i})\fbm{s}_{i}=\fbm{f}^{*}_{i}\textrm{,}
\end{aligned}
\end{equation}
with $\fbm{f}^{*}_{i}=-\bm{\Phi}_{i}(\fbm{q}_{i},\dotbm{q}_{i},\fbm{e}_{i},\dotbm{e}_{i})\tilfbm{\theta}_{i}-\mu_{i}\fbm{s}_{i}-\kappa_{i}\dotbm{q}_{i}$ and $\tilfbm{\theta}_{i}=\bm{\theta}_{i}-\hfbm{\theta}_{i}$. Hence, the auxiliary variables $\fbm{s}_{i}$ in~\eqref{equ23} transform the closed-loop agent dynamics from~\eqref{equ18} to~\eqref{equ24}. The dynamics~\eqref{equ24} inherit the passivity of~\eqref{equ18} because $\dotbm{M}_{i}(\fbm{q}_{i})-2\fbm{C}_{i}(\fbm{q}_{i},\dotbm{q}_{i})$ are skew-symmetric. 

Algebraic manipulation using $\dotbm{q}_{i}=-\alpha\fbm{e}_{i}+\fbm{s}_{i}$ and $\fbm{e}_{i}=\sum_{j\in\mathcal{N}_{i}(0)}\nabla_{i}\psi(\|\fbm{q}_{ij}\|)$ leads to:
\begin{equation}\label{equ25}
\dotbm{q}=-\alpha\left(\fbm{L}(\fbm{q})\otimes\fbm{I}_{n}\right)\fbm{q}+\fbm{s}\textrm{,}
\end{equation}
where $\fbm{q}=\begin{pmatrix}\tbm{q}_{1},\cdots,\tbm{q}_{N}\end{pmatrix}^\mathsf{T}$ and $\fbm{s}=\begin{pmatrix}\tbm{s}_{1},\cdots,\tbm{s}_{N}\end{pmatrix}^\mathsf{T}$ stack the configuration and the auxiliary variables, respectively, and $\fbm{L}(\fbm{q})$ has been defined in~\sect{sec: single-integrator}.

\begin{remark}
\normalfont The control law~\eqref{equ23} converts the closed-loop agent dynamics~\eqref{equ18} into the cascaded interconnection of the passive dynamics~\eqref{equ24} and~\eqref{equ25}.
The output $\fbm{s}$ of the dynamics~\eqref{equ24} is input to the dynamics~\eqref{equ25}. Because the dynamics~\eqref{equ25} are input-to-state stable with input $\fbm{s}$, making $\fbm{s}\to\fbm{0}$ through the control of~\eqref{equ24} is sufficient to guarantee the coordination of the Euler-Lagrange MAS.
\end{remark}

The Lyapunov candidate function, used to analyze the connectivity-preserving consensus of Euler-Lagrange MAS-s with uncertain dynamics in this section, is:
\begin{equation}\label{equ26}
V=\frac{1}{2}\sum^{N}_{i=1}\left[\frac{1}{\alpha\kappa_{i}}\left(\tbm{s}_{i}\fbm{M}_{i}\fbm{s}_{i}+\ttilfbm{\theta}_{i}\tilfbm{\theta}_{i}\right)+\sum_{j\in\mathcal{N}_{i}(0)}\psi(\|\fbm{q}_{ij}\|)\right]\textrm{.}
\end{equation}

After using~\eqref{equ23} and~\eqref{equ24} and the definitions of $\fbm{s}_{i}$ and $\fbm{e}_{i}$, the derivative of $V$ can be written in the form:
\begin{align*}
\dot{V}=&\sum^{N}_{i=1}\Big(\frac{1}{\alpha\kappa_{i}}\ttilfbm{\theta}_{i}\left(-\tfbm{\Phi}_{i}(\fbm{q}_{i},\dotbm{q}_{i},\fbm{e}_{i},\dotbm{e}_{i})\fbm{s}_{i}-\dothatfbm{\theta}_{i}\right)-\frac{1}{\alpha}\dottbm{q}_{i}\dotbm{q}_{i}\\
&-\frac{\mu_{i}}{\alpha\kappa_{i}}\tbm{s}_{i}\fbm{s}_{i}-\dottbm{q}_{i}\fbm{e}_{i}\Big)+\sum^{N}_{i=1}\dottbm{q}_{i}\sum_{j\in\mathcal{N}_{i}(0)}\nabla_{i}\psi(\|\fbm{q}_{ij}\|)\\
=&-\frac{1}{\alpha}\sum^{N}_{i=1}\left(\frac{\mu_{i}}{\kappa_{i}}\tbm{s}_{i}\fbm{s}_{i}+\dottbm{q}_{i}\dotbm{q}_{i}\right)\leq 0\textrm{.}
\end{align*}
Because $V\geq\Psi(\fbm{q})$, the third property of $\Psi(\fbm{q})$ together with $\dot{V}\leq 0$ imply that $\Psi(t)\leq V(t)\leq V(0)$ and that the initial connectivity is preserved if $V(0)<\Psi_{max}$. 

\ass{ass2} implies that $\Psi(0)<\Psi_{max}$. If $\Psi(\fbm{q})$ is unbounded, then the condition that $V(0)<\Psi_{max}=\infty$ is obviously guaranteed. If $\Psi(\fbm{q})$ is bounded, then $V(0)<\Psi_{max}$ can be guaranteed by rendering $\frac{1}{\alpha\kappa_{i}}\left(\tbm{s}_{i}\fbm{M}_{i}\fbm{s}_{i}+\ttilfbm{\theta}_{i}\tilfbm{\theta}_{i}\right)$ sufficiently small at $t=0$. Because $\tbm{s}_{i}\fbm{M}_{i}\fbm{s}_{i}+\ttilfbm{\theta}_{i}\tilfbm{\theta}_{i}$ depend on the initial state of the MAS, $\alpha\kappa_{i}$ should be chosen sufficiently large to make $V(0)<\Psi_{max}$. Tuning $\alpha$ is not straightforward because the auxiliary variables $\fbm{s}_{i}=\dotbm{q}_{i}+\alpha\fbm{e}_{i}$ depend on it. Instead, it is simpler to select $\kappa_{i}$ sufficiently large to guarantee $V(0)<\Psi_{max}$ because both $\fbm{s}_{i}$ and $\tilfbm{\theta}_{i}$ are independent of $\kappa_{i}$ .

\begin{remark}
\normalfont By the analysis above, the distributed control~\eqref{equ23} preserves the connectivity of a fully actuated Euler-Lagrange MAS with uncertain dynamics whether the gradient-based terms $\fbm{e}_{i}=\sum_{j\in\mathcal{N}_{i}(0)}\nabla_{i}\psi(\|\fbm{q}_{ij}\|)$ in the auxiliary variables $\fbm{s}_{i}$ derived from an unbounded or from a bounded generalized potential function~\eqref{equ1}. An unbounded potential function generates unbounded gradient-based terms $-\mu_{i}\fbm{s}_{i}$ in the controls $\fbm{f}^{*}_{i}$. Thus, it stiffens indefinitely the couplings between neighbouring agents with endangered communications. Sufficiently large $\kappa_{i}$ inject enough damping $-\kappa_{i}\dotbm{q}_{i}$ to stop neighbouring agents from moving away from each other while they are still in their communication distance.
\end{remark}

Further, the derivative of $V$ leads to the conclusions that $\fbm{s}_{i}\in\mathcal{L}_{2}\cap\mathcal{L}_{\infty}$ and $\tilfbm{\theta}_{i}\in\mathcal{L}_{\infty}$, which, together with the system dynamics, lead to $\dotbm{s}_{i}\in\mathcal{L}_{\infty}$ and, thus, to $\fbm{s}_{i}\to\fbm{0}$ as $t\to\infty$. Then, the analysis similar to~\cite{Nuno2011TAC} leads to the conclusion that $\fbm{q}_{1}\to\cdots\to\fbm{q}_{N}$ and coordination is achieved.

The analysis up to here has proven that: (i) no controller can drive an Euler-Lagrange MAS with bounded actuation to connectivity-preserving consensus from any initial state; and (ii) a gradient-based controller derived from the potential function~\eqref{equ1} can drive an Euler-Lagrange MAS with full actuation to connectivity-preserving consensus from any initial state, whether the potential is bounded or unbounded. The following analysis develops a framework to overcome the intrinsic conflict between actuator saturation and connectivity maintenance for Euler-Lagrange MAS-s which start from rest.

\subsection{Actuator Saturation}\label{sec: saturations}

This section proves that an indirect coupling framework based on the generalized potential function~\eqref{equ1} drives to connectivity-preserving consensus Euler-Lagrange MAS-s that have bounded actuation and start from rest. 

The proposed indirect coupling framework is designed as the control: 
\begin{equation}\label{equ27}
\begin{aligned}
\hbm{f}_{i}=&\Sat^{p}_{i}\left(p_{i}(\hbm{q}_{i}-\fbm{q}_{i})\right)+\fbm{g}_{i}\textrm{,}\\
\ddothbm{q}_{i}=&\Sat^{p}_{i}\left(p_{i}(\fbm{q}_{i}-\hbm{q}_{i})\right)-\sum_{j\in\mathcal{N}_{i}(0)}\nabla_{i}\psi(\|\hbm{q}_{ij}\|)-k_{i}\dothatbm{q}_{i}\textrm{,}
\end{aligned}
\end{equation}
where: $i=1,\cdots,N$; $\hbm{f}_{i}$ is the designed control of agent $i$ of the MAS in~\eqref{equ18}; $\hbm{q}_{i}$ is the position of proxy $i^{'}$ of agent $i$; $\tilbm{q}_{i}=\fbm{q}_{i}-\hbm{q}_{i}$ is the position mismatch between agent $i$ and its proxy $i^{'}$; $\hbm{q}_{ij}=\hbm{q}_{i}-\hbm{q}_{j}$ is the displacement between proxies $i^{'}$ and $j^{'}$;  as in~\sect{sec: single-integrator}, $\Sat^{p}_{i}(\cdot)$ is the standard saturation function, with bounds to be determined; and $k_{i}$ and $p_{i}$ are positive constants. Thus, in the proposed framework, agent $i$ is coupled to its proxy $i^{'}$ through saturated P control, which accounts for the physical limits of the MAS actuators. Proxy $i^{'}$ is coupled to agent $i$ and to the proxies $j^{'}$ of all agents $j$ adjacent to agent $i$. The inter-proxy couplings are gradient-based control terms derived from the potential function~\eqref{equ1}, and can be unbounded and arbitrarily stiff because the proxies are virtual dynamics introduced through design.

\begin{assumption}\label{ass3}
The actuators can more than balance gravity throughout the workspace. That is, there exist $\bm{\gamma}_{i}$,~$i=1,...,N$, such that $|g^{k}_{i}|\leq \gamma^{k}_{i}<\bar{f}^{k}_{i}\quad \forall k=1,\cdots,n$ and $\forall \bm{q}_{i}$, where $\bar{\fbm{f}}_{i}$ is the maximum actuation of agent $i$. 
\end{assumption}

\begin{remark}
\normalfont In~\eqref{equ27}, $\Sat^{p}_{i}(\cdot)$ is selected to prevent the saturation of the actual control $\bm{f}_{i}$ and to make the agent-proxy coupling passive in $\ddothbm{q}_{i}$. Choosing the bound of $\Sat^{p}_{i}(\cdot)$ to be $\overline{\fbm{f}}_{i}-\bm{\gamma}_{i}$, the control force applied on agent $i$ is $\fbm{f}_{i}=\Sat_{i}\big(\hbm{f}_{i}\big)=\hbm{f}_{i}$ and the potential energy stored in the agent-proxy coupling is:
\begin{align}\label{equ28}
\phi_{i}(\tilbm{q}_{i})=\int^{\tilbm{q}_{i}}_{\fbm{0}}\Sat^{p}_{i}(p_{i}\bm{\sigma})d\bm{\sigma}\textrm{.}
\end{align}
To preserve connectivity, the saturated agent-proxy couplings and the unconstrained inter-proxy couplings should be designed to guarantee that $d_{ij}(t)=\|\fbm{q}_{ij}\|\leq r$ $\forall(i,j)\in\mathcal{E}(0)$ $\forall t\ge 0$. From~\eqref{equ27} and the triangle inequality, it suffices to guarantee that $\|\tilbm{q}_{i}\|+\|\hbm{q}_{ij}\|+\|\tilbm{q}_{j}\|\leq r$ $\forall t\ge 0$. \ass{ass2} implies that $\|\fbm{q}_{ij}(0)\|<r-\epsilon<r$ for any pair of initially adjacent agents. Therefore, after choosing $\hbm{q}_{i}(0)=\fbm{q}_{i}(0)$ and $\dothatbm{q}_{i}(0)=\fbm{0}$, the connectivity of the Euler-Lagrange MAS with bounded actuation can be guaranteed by enforcing $\|\hbm{q}_{ij}(t)\|\leq r-\frac{2\epsilon}{3}=\hat{r}$ and $\|\tilbm{q}_{i}(t)\|\leq\frac{\epsilon}{3}$ for all $t\ge 0$, $(i,j)\in\mathcal{E}(0)$ and $i=1,\cdots,N$.
\end{remark}

\begin{remark}\label{rem11}
\normalfont \ass{ass2} also implies that $\|\hbm{q}_{ij}(0)\|=\|\fbm{q}_{ij}(0)\|<r-\epsilon<\hat{r}$. Therefore, $\|\hbm{q}_{ij}(t)\|\leq \hat{r}$ can be guaranteed by bounding a potential function $\psi(\|\hbm{q}_{ij}\|)$ by:
\begin{align*}
\psi(\|\hbm{q}_{ij}\|)=\frac{\rho\|\hbm{q}_{ij}\|^{2}}{\hat{r}^{2}-\|\hbm{q}_{ij}\|^{2}+Q}\leq\frac{\rho\hat{r}^{2}}{Q}=\Psi_{max}\textrm{,}
\end{align*}
with $\rho$ and $Q$ positive constants and $\left(N(r-\epsilon)^{2}-\hat{r}^{2}\right)Q\leq\hat{r}^{2}\left(\hat{r}^{2}-(r-\epsilon)^{2}\right)$. Then it follows that $\Psi(\hbm{q}(0))\leq\Psi_{max}$ and that $\|\hbm{q}_{ij}(t)\|\leq\hat{r}$ $\forall t\ge 0$ if $\Psi(\hbm{q}(t))\leq\Psi(\hbm{q}(0))$ $\forall t\ge 0$.
\end{remark}

The function used to investigate the connectivity-preserving consensus of Euler-Lagrange MAS-s with limited actuation is:
\begin{equation}\label{equ29}
\begin{aligned}
V=&\sum^{N}_{i=1}\frac{1}{2}\left(\dottbm{q}_{i}\fbm{M}_{i}\dotbm{q}_{i}+\dothattbm{q}_{i}\dothatbm{q}_{i}\right)+\frac{1}{2}\sum^{N}_{i=1}\sum_{j\in\mathcal{N}_{i}(0)}\psi(\|\hbm{q}_{ij}\|)\\
&+\sum^{N}_{i=1}\int^{\tilbm{q}_{i}}_{\fbm{0}}\Sat^{p}_{i}(p_{i}\bm{\sigma})d\bm{\sigma}\textrm{.}
\end{aligned}
\end{equation} 
By~\eqref{equ18} and~\eqref{equ27}, the derivative of $V$ is:
\begin{align*}
\dot{V}=&\sum^{N}_{i=1}\dottbm{q}_{i}\Sat^{p}_{i}\left(p_{i}(\hbm{q}_{i}-\fbm{q}_{i})\right)+\sum^{N}_{i=1}\dothattbm{q}_{i}\Sat^{p}_{i}\left[p_{i}(\fbm{q}_{i}-\hbm{q}_{i})\right]\\
&-\sum^{N}_{i=1}\dothattbm{q}_{i}\sum_{j\in\mathcal{N}_{i}(0)}\nabla_{i}\psi(\|\hbm{q}_{ij}\|)-\sum^{N}_{i=1}k_{i}\dothattbm{q}_{i}\dothatbm{q}_{i}\\
&+\frac{1}{2}\sum^{N}_{i=1}\sum_{j\in\mathcal{N}_{i}(0)}\left(\dothattbm{q}_{i}\nabla_{i}\psi(\|\hbm{q}_{ij}\|)+\dothattbm{q}_{j}\nabla_{j}\psi(\|\hbm{q}_{ij}\|)\right)\\
&+\dotttilbm{q}_{i}\sum^{N}_{i=1}\Sat^{p}_{i}(p_{i}\tilbm{q}_{i})\textrm{.}
\end{align*}
Using $\dottilbm{q}_{i}=\dotbm{q}_{i}-\dothatbm{q}_{i}$ and symmetry to rearrange terms:
\begin{align*}
&\sum^{N}_{i=1}\sum_{j\in\mathcal{N}_{i}(0)}\dothattbm{q}_{j}\nabla_{j}\psi(\|\hbm{q}_{ij}\|)=\sum^{N}_{j=1}\sum_{i\in\mathcal{N}_{j}(0)}\dothattbm{q}_{i}\nabla_{i}\psi(\|\hbm{q}_{ji}\|)\\
=&\sum^{N}_{j=1}\sum_{i\in\mathcal{N}_{j}(0)}\dothattbm{q}_{i}\nabla_{i}\psi(\|\hbm{q}_{ij}\|)=\sum^{N}_{i=1}\sum_{j\in\mathcal{N}_{i}(0)}\dothattbm{q}_{i}\nabla_{i}\psi(\|\hbm{q}_{ij}\|)\textrm{,}
\end{align*}
it follows that:
\begin{equation}\label{equ30}
\dot{V}=-\sum^{N}_{i=1}k_{i}\dothattbm{q}_{i}\dothatbm{q}_{i}\leq 0\textrm{,}
\end{equation}
which implies that $V(t)\leq V(0)$ for any $t\geq 0$ if the initial connectivity is maintained.

\sect{sec: conflict} has proven that no controller can drive to connectivity-preserving consensus an Euler-Lagrange MAS with bounded actuation from any initial state. Because the initial velocities are the culprit, this paper will synchronize with connectivity maintenance MAS-s which start from rest. Thus, the remainder of the paper uses following assumption:
\begin{assumption}\label{ass4}
The Euler-Lagrange MAS with bounded actuation is initially at rest, i.e., $\dotbm{q}_{i}(0)=\fbm{0}$ for any $i=1,\cdots,N$.
\end{assumption}

Initially, the agent-proxy couplings store no potential energy:
\begin{align*}
\phi_{i}(\tilbm{q}_{i}(0))=\int^{\tilbm{q}_{i}(0)}_{\fbm{0}}\Sat^{p}_{i}(p_{i}\bm{\sigma})d\bm{\sigma}=0\textrm{,}
\end{align*}
because $\tilbm{q}_{i}(0)=\fbm{q}_{i}(0)-\hbm{q}_{i}(0)=\fbm{0}$, and the MAS is at rest, so $\dottbm{q}_{i}\fbm{M}_{i}\dotbm{q}_{i}+\dothattbm{q}_{i}\dothatbm{q}_{i}=0$ and $\dothatbm{q}_{i}(0)=\fbm{0}$. After substitution of all terms in~\eqref{equ29}, it follows that $V(0)=\Psi(\hbm{q}(0))=\Psi(\fbm{q}(0))$.  

\begin{lemma}
The potential function $\phi_{i}(\tilbm{q}_{i})$ in~\eqref{equ28} is convex with respect to $\tilbm{q}_{i}$ on $\mathbb{R}^{n}$.
\end{lemma}
\begin{proof}
The gradient of $\phi_{i}(\tilbm{q}_{i})$ with respect to $\tilbm{q}_{i}$ is 
\begin{align*}
\nabla\phi_{i}(\tilbm{q}_{i})=\Sat^{p}_{i}(p_{i}\tilbm{q}_{i})\textrm{.}
\end{align*}
Let $\fbm{x}\in\mathbb{R}^{n}$ and $\fbm{y}\in\mathbb{R}^{n}$. Then,
\begin{align*}
&\left(\nabla\phi_{i}(\fbm{x})-\nabla\phi(\fbm{y})\right)^\mathsf{T}(\fbm{x}-\fbm{y})\\
=&\left(\Sat^{p}_{i}(p_{i}\fbm{x})-\Sat^{p}_{i}(p_{i}\fbm{y})\right)^\mathsf{T}(\fbm{x}-\fbm{y})\\
=&\sum^{n}_{k=1}\left(\sat^{p}_{ik}(p_{i}x_{k})-\sat^{p}_{ik}(p_{i}y_{k})\right)(x_{k}-y_{k})\geq 0\textrm{,}
\end{align*}
because $\sat^{p}_{ik}(\cdot)$ are increasing functions. By the first-order convexity condition, $\phi_{i}(\tilbm{q}_{i})$ is convex on $\mathbb{R}^{n}$.
\end{proof}

\begin{lemma}
On $B(\fbm{0},\frac{\epsilon}{3})=\{\tilbm{q}_{i}\ |\ \|\tilbm{q}_{i}\|\leq\frac{\epsilon}{3}\}$, the potential function $\phi_{i}(\tilbm{q}_{i})$ is maximum on the boundary $\|\tilbm{q}_{i}\|=\frac{\epsilon}{3}$ and minimum at the origin $\|\tilbm{q}_{i}\|=\fbm{0}$. 
\end{lemma}
\begin{proof}
The potential function $\phi_{i}(\tilbm{q}_{i})$ is continuous on the closed and bounded ball $B(\fbm{0},\frac{\epsilon}{3})\subseteq\mathbb{R}^{n}$. Therefore, by the Weierstrass theorem, $\phi_{i}(\tilbm{q}_{i})$ has its global minimum and maximum on $B(\fbm{0},\frac{\epsilon}{3})$. Further, $\phi_{i}(\tilbm{q}_{i})$ is convex on $\mathbb{R}^{n}$. Hence, on the ball $B(\fbm{0},\frac{\epsilon}{3})$, $\phi_{i}(\tilbm{q}_{i})$ attains its global maximum on the boundary of $B(\fbm{0},\frac{\epsilon}{3})$, and its global minimum at the point with $\nabla\phi_{i}(\tilbm{q}_{i})=\Sat^{p}_{i=1}(p_{i}\tilbm{q}_{i})=\fbm{0}$, that is at $\tilbm{q}_{i}=\fbm{0}$. 
\end{proof}

\begin{lemma}\label{lem4}
Let $\phi^{*}_{i}$ be the minimum value of the potential function $\phi_{i}$ on the boundary of the ball $B(\fbm{0},\frac{\epsilon}{3})$:
\begin{equation}\label{equ31}
\begin{aligned}
\phi^{*}_{i}=&\min \limits_{\tilbm{q}_{i}}\quad \phi_{i}(\tilbm{q}_{i})=\int^{\tilbm{q}_{i}}_{\fbm{0}}\Sat^{p}_{i}(p_{i}\bm{\sigma})d\bm{\sigma}\\
&s.t.\quad \quad \|\tilbm{q}_{i}\|=\frac{\epsilon}{3}\textrm{.}
\end{aligned} 
\end{equation}
If $\phi_{i}(\tilbm{q}_{i})\leq\phi^{*}_{i}$, then $\tilbm{q}_{i}\in B(\fbm{0},\frac{\epsilon}{3})$.
\end{lemma}
\begin{proof}
Suppose there exists a $\tilbm{q}_{i}\notin B(\fbm{0},\frac{\epsilon}{3})$ such that $\phi_{i}(\tilbm{q}_{i})\leq\phi^{*}_{i}$. Let $\fbm{x}$ be such that $\phi_{i}(\fbm{x})=\phi^{*}_{i}$. Then, it follows that $\fbm{x}=\lambda\fbm{0}+(1-\lambda)\tilbm{q}_{i}$ for some $0<\lambda<1$ and, by the convexity of $\phi_{i}(\tilbm{q}_{i})$, that:
\begin{align*}
\phi_{i}(\fbm{x})=&\phi_{i}\left(\lambda\fbm{0}+(1-\lambda)\tilbm{q}_{i}(t)\right)\leq\lambda\phi_{i}(\fbm{0})+(1-\lambda)\phi_{i}(\tilbm{q}_{i})\\
=&(1-\lambda)\phi_{i}(\tilbm{q}_{i}(t))<\phi^{*}_{i}\textrm{,}
\end{align*}
which contradicts $\phi_{i}(\fbm{x})=\phi^{*}_{i}$. Therefore, $\phi_{i}(\tilbm{q}_{i})\leq\phi^{*}_{i}$ implies that $\tilbm{q}_{i}\in B(\fbm{0},\frac{\epsilon}{3})$, see \fig{fig9}.
\end{proof}

\begin{figure}[!hbt]
\centering
\includegraphics[width=8cm,height=5cm]{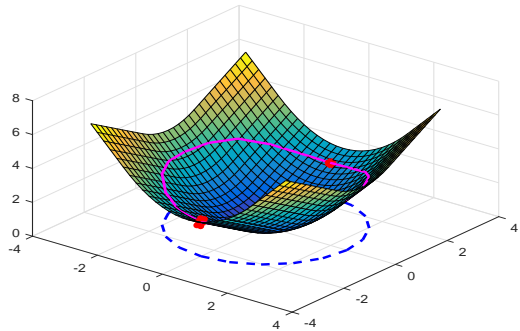}
\caption{A convex potential function $\phi_{i}(\tilbm{q}_{i})$ on $\mathbb{R}^{2}$ whose values on the boundary of a ball around the origin of $\mathbb{R}^{2}$~(blue dashed circle) are shown as the pink curve. Along the pink curve, there are two points~(red points) where the potential function is equal to $\phi^{*}_{i}$. All points $\tilbm{q}_{i}$ that satisfy $\phi_{i}(\tilbm{q}_{i})<\phi^{*}_{i}$ are in the blue circle.}
\label{fig9}
\end{figure} 

Let $\phi^{*}=\min \limits_{i=1\cdots,N}(\phi^{*}_{i})$, and the potential function $\Psi(\hbm{q})$~\eqref{equ29} satisfy the first two properties of~\eqref{equ1} and attain its maximum $\Psi_{max}=\phi^{*}$ if $\exists (i,j)\in\mathcal{E}(0)$ such that $\|\hbm{q}_{ij}\|=\hat{r}$. \rem{rem11} shows that $\Psi(\hbm{q})$ exists. Then, $V(0)=\Psi(\hbm{q}(0))<\Psi_{max}=\phi^{*}$ together with~\eqref{equ30} imply that $V(t)<\phi^{*}$ and, further, that $\Psi(\hbm{q}(t))<\phi^{*}$ and $\phi_{i}(\tilbm{q}_{i}(t))<\phi^{*}\leq\phi^{*}_{i}$, for all $i=1,\cdots,N$. It follows that $\|\hbm{q}_{ij}(t)\|<\hat{r}$ for all $(i,j)\in\mathcal{E}(0)$ and $\|\tilbm{q}_{i}(t)\|<\frac{\epsilon}{3}$ for $i=1,\cdots,N$, and therefore that $\|\fbm{q}_{ij}(t)\|\leq\|\tilbm{q}_{i}\|+\|\hbm{q}_{ij}\|+\|\tilbm{q}_{j}\|<r$, i.e., connectivity is preserved.

Coordination can be concluded noting that~\eqref{equ30} leads to $\dothatbm{q}_{i}\in\mathcal{L}_{2}\cap\mathcal{L}_{\infty}$, which yields that $\dothatbm{q}_{i}\to\fbm{0}$. After taking the derivative of $\ddothbm{q}_{i}$ in~\eqref{equ27}, Barbalat's lemma yields $\ddothbm{q}_{i}\to\fbm{0}$. Bounded second derivatives of $\ddothbm{q}_{i}$ lead to $\dddothbm{q}_{i}\to\fbm{0}$ and further $\dothatbm{q}_{i}\to\dotbm{q}_{i}\to\fbm{0}$. The derivative of~\eqref{equ18} shows that $\ddotbm{q}_{i}\to\fbm{0}$ and thus $\Sat^{p}_{i}\left(p_{i}(\hbm{q}_{i}-\fbm{q}_{i})\right)\to\fbm{0}$. Together, all the above inferences lead to $\tilbm{q}_{i}\to\fbm{0}$ and $\hbm{q}_{ij}\to\fbm{0}$, i.e., $\fbm{q}_{i}\to\fbm{q}_{j}$. 

\begin{remark}
\normalfont The indirect coupling control~\eqref{equ27} has three main benefits: 1) it does not require velocity estimation because $\hfbm{f}_{i}$ use only the position of the agent and its proxy; 2) it decomposes connectivity preservation and actuator saturation into two subproblems that can be addressed separately and neither of which needs to consider multiple couplings with limited actuation; 3) it can also cope with time-varying delays and system uncertainties, as shown in the next subsection.
\end{remark}

\subsection{Uncertain parameters and communication delays}

Practical Euler-Lagrange MAS-s may have uncertain parameters and time-varying delays in the inter-agent communications. If parameters are uncertain, the agents cannot compensate gravity directly. If the communications are delayed, the agents cannot receive the positions of their neighbours instantly.
Given two agents $i$ and $j$, adjacent at time $t$. If the transmission from agent $i$ to agent $j$ has a delay $T_{ji}(t)$, then the two agents can receive only the delayed positions $\fbm{q}_{jd}(t)=\fbm{q}_{j}(t-T_{ji}(t))$ and $\fbm{q}_{id}(t)=\fbm{q}_{i}(t-T_{ij}(t))$, respectively. Then, the control should keep the agents within communication distance $d_{ij}(t)=\|\fbm{q}_{ij}(t)\|\leq r$ using only the delayed positions of neighbours. 

The following indirect coupling control strategy is designed to maintain all initial communication links and to synchronize the Euler-Lagrange MAS using only limited actuation: 
\begin{equation}\label{equ32}
\begin{aligned}
\hbm{f}_{i}=&\fbm{\Phi}_{i}(\fbm{q}_{i},\dotbm{q}_{i},\fbm{e}_{i},\dotbm{e}_{i})\hfbm{\theta}_{i}-\Sat^{pd}_{i}(\mu_{i}\fbm{s}_{i}+\fbm{e}_{i})\textrm{,}\\
\dothatfbm{\theta}_{i}=&\text{Proj}_{\hfbm{\theta}_{i}}(\bm{\omega}_{i})\textrm{,}\\
\bm{\omega}_{i}=&-\beta_{i}\tfbm{\Phi}_{i}(\fbm{q}_{i},\dotbm{q}_{i},\fbm{e}_{i},\dotbm{e}_{i})\fbm{s}_{i}\textrm{,}\\
\ddothbm{q}_{i}=&\Sat^{pd}_{i}(\fbm{e}_{i})-\sum_{j\in\mathcal{N}_{i}(0)}\nabla_{i}\psi(\|\hbm{q}_{ijd}\|)-k_{i}\dothatbm{q}_{i}\textrm{,}
\end{aligned}
\end{equation}
where: $i=1,\cdots,N$; $\hbm{f}_{i}$ is the designed control force of agent $i$; $\bm{\Phi}_{i}(\fbm{q}_{i},\dotbm{q}_{i},\fbm{e}_{i},\dotbm{e}_{i})\hfbm{\theta}_{i}$ is the adaptive dynamics compensation as that in~\eqref{equ23}; $\fbm{e}_{i}=p_{i}\tilbm{q}_{i}=p_{i}(\fbm{q}_{i}-\hbm{q}_{i})$ is the scaled displacement between agent $i$ and its proxy; $\fbm{s}_{i}=\dotbm{q}_{i}+\alpha\fbm{e}_{i}$; $\Sat^{pd}_{i}(\cdot)$ is the standard saturation function as defined in~\sect{sec: single-integrator}; $\hbm{q}_{ijd}=\hbm{q}_{i}(t)-\hbm{q}_{jd}(t)$ with $\hbm{q}_{jd}(t)=\hbm{q}_{j}(t-T_{ji}(t))$; and $\mu_{i}$, $p_{i}$, $k_{i}$ and $\alpha$ are positive constants. 


Component-wise, the projection operators $\text{Proj}_{\hfbm{\theta}_{i}}(\bm{\omega}_{i})$ are:
\begin{equation}\label{equ33}
\begin{aligned}
\dot{\hat{\theta}}^{k}_{i}=\begin{cases}
\left(1-\upsilon_{l}(\hat{\theta}^{k}_{i})\right)\omega^{k}_{i}\quad &\underline{\theta}^{k}_{i}\leq\hat{\theta}^{k}_{i}\leq\underline{\theta}^{k}_{i}+\delta\ \&\ \omega^{k}_{i}<0\textrm{,}\\
\left(1-\upsilon_{u}(\hat{\theta}^{k}_{i})\right)\omega^{k}_{i}\quad &\overline{\theta}^{k}_{i}-\delta\leq\hat{\theta}^{k}_{i}\leq\overline{\theta}^{k}_{i}\ \&\ \omega^{k}_{i}>0\textrm{,}\\
\omega^{k}_{i}\quad &\text{otherwise}\textrm{,}
\end{cases}
\end{aligned}
\end{equation}
where $\upsilon_{l}(\hat{\theta}^{k}_{i})=\min\left(1,\frac{\underline{\theta}^{k}_{i}+\delta-\hat{\theta}^{k}_{i}}{\delta}\right)$ and $\upsilon_{u}(\hat{\theta}^{k}_{i})=\min\left(1,\frac{\hat{\theta}^{k}_{i}-\overline{\theta}^{k}_{i}+\delta}{\delta}\right)$, with $0<\delta<\frac{1}{2}(\overline{\theta}^{k}_{i}-\underline{\theta}^{k}_{i})$ and $k=1,\cdots,n$. From~\cite{Krstic1995}, the projectors guarantee that the selection $\underline{\theta}^{k}_{i}\leq\hat{\theta}^{k}_{i}(0)\leq\overline{\theta}^{k}_{i}$ leads to $\underline{\theta}^{k}_{i}\leq\hat{\theta}^{k}_{i}\leq\overline{\theta}^{k}_{i}$ for all $k=1,\cdots,n$. Then, the dynamics compensation terms are bounded as in the following lemma.
\begin{lemma}
If $\underline{\theta}^{k}_{i}\leq\hat{\theta}^{k}_{i}\leq\overline{\theta}^{k}_{i}$ and $|\dot{q}^{k}_{i}|\leq{v}^{k}_{i}$ and $|\dot{\hat{q}}^{k}_{i}|\leq\hat{v}^{k}_{i}$ for some constants ${v}^{k}_{i}$ and $\hat{v}^{k}_{i}$, then there exist $\eta^{k}_{i}\geq{0}$ such that the dynamic compensation terms are bounded by $|\hat{g}^{k}_{i}|\leq\eta^{k}_{i}$, where $k=1,\cdots,n$ and $\begin{pmatrix}\hat{g}^{1}_{i} \cdots \hat{g}^{k}_{i} \cdots \hat{g}^{n}_{i}\end{pmatrix}^\mathsf{T}=\bm{\Phi}_{i}(\fbm{q}_{i},\dotbm{q}_{i},\fbm{e}_{i},\dotbm{e}_{i})\hfbm{\theta}_{i}$.
\end{lemma}
The proof of the lemma follows from the definition of $\bm{\Phi}_{i}(\fbm{q}_{i},\dotbm{q}_{i},\fbm{e}_{i} ,\dotbm{e}_{i})\hfbm{\theta}_{i}$ directly, and is omitted here. 

Feasibility of the connectivity-preserving consensus problem for Euler-Lagrange MAS-s with limited actuation also requires \ass{ass4} and \ass{ass5} below.
\begin{assumption}\label{ass5}
The actuators can more than execute the dynamic compensations, i.e., ${\eta}^{k}_{i}<\overline{{f}}^{k}_{i}$ for all $k=1,\cdots,n$.
\end{assumption}
After choosing the bounds of $\Sat^{pd}_{i}(\cdot)$ equal to $\overline{\fbm{f}}_{i}-\bm{\eta}_{i}$, the actual agent controls are equal to their designed controls, $\fbm{f}_{i}=\Sat_{i}(\hbm{f}_{i})=\hbm{f}_{i}$.

\begin{remark}
\normalfont As in~\sect{sec: saturations}, saturated P controls couple each agent $i$ to its proxy $i^{'}$ in~\eqref{equ32}, and gradient-based controls couple proxies of adjacent agents. The virtual proxies facilitate a design 
which addresses uncertainties and time-varying delays in $\hbm{f}_{i}$ and $\ddothbm{q}_{i}$ separately. In $\hbm{f}_{i}$, system uncertainties are compensated by $\bm{\Phi}_{i}(\fbm{q}_{i},\dotbm{q}_{i},\fbm{e}_{i},\dotbm{e}_{i})\hfbm{\theta}_{i}$, which rely only on local information $\fbm{q}_{i}$, $\hbm{q}_{i}$, $\dotbm{q}_{i}$ and $\dothatbm{q}_{i}$. In turn, the time-varying communication delays distort only the inter-proxy couplings in $\ddothbm{q}_{i}$. Because the proxies are designed dynamics with no physical constraints, sufficient damping $-\kappa_{i}\dothatbm{q}_{i}$ can be injected in the inter-proxy couplings to suppress the distorsions introduced by the time-varying communication delays.
\end{remark}

The function which serves to investigate the connectivity preservation and synchronization of an Euler-Lagrange MAS that starts from rest and has bounded actuation, uncertain parameters and time-varying communication delays is:
\begin{equation}\label{equ34}
\begin{aligned}
V=&\frac{1}{2}\sum^{N}_{i=1}\left(\tbm{s}_{i}\fbm{M}_{i}\fbm{s}_{i}+\frac{1}{\beta_{i}}\ttilfbm{\theta}_{i}\tilfbm{\theta}_{i}\right)+\sum^{N}_{i=1}\int^{\tilbm{q}_{i}}_{\fbm{0}}\Sat^{pd}_{i}(p_{i}\bm{\sigma})d\bm{\sigma}\\
&+\frac{1}{2}\sum^{N}_{i=1}\dothattbm{q}_{i}\dothatbm{q}_{i}+\frac{1}{2}\sum^{N}_{i=1}\sum_{j\in\mathcal{N}_{i}(0)}\psi(\|\hbm{q}_{ij}\|)\textrm{,}
\end{aligned}
\end{equation}
where: $\beta_{i}>0$; $\tilfbm{\theta}_{i}=\bm{\theta}_{i}-\hfbm{\theta}_{i}$ are the estimation errors for the system parameters; and, as in~\sect{sec: saturations}, 
\begin{align*}
\phi_{i}(\tilbm{q}_{i})=\int^{\tilbm{q}_{i}}_{\fbm{0}}\Sat^{pd}_{i}(p_{i}\bm{\sigma})d\bm{\sigma}
\end{align*}
are convex potential functions with global maximum on the boundary of $B(\fbm{0},\frac{\epsilon}{3})$ and with minimum at $\|\tilbm{q}_{i}\|=\fbm{0}$.

Similar to~\eqref{equ24}, the closed-loop dynamics of the uncertain Euler-Lagrange MAS~\eqref{equ18} with bounded actuation under the control~\eqref{equ32} are:
\begin{align*}
&\fbm{M}_{i}(\fbm{q}_{i})\dotbm{s}_{i}+\fbm{C}_{i}(\fbm{q}_{i},\dotbm{q}_{i})\fbm{s}_{i}\\
&=-\bm{\Phi}_{i}(\fbm{q}_{i},\dotbm{q}_{i},\fbm{e}_{i},\dotbm{e}_{i})\tilfbm{\theta}_{i}-\Sat^{pd}_{i}(\mu_{i}\fbm{s}_{i}+\fbm{e}_{i})\textrm{.}
\end{align*} 
From~\cite{Krstic1995} and the projector properties, it follows that:
\begin{align*}
&-\tfbm{s}_{i}\bm{\Phi}_{i}(\fbm{q}_{i},\dotbm{q}_{i},\fbm{e}_{i},\dotbm{e}_{i})\tilfbm{\theta}_{i}+\frac{1}{\beta_{i}}\ttilfbm{\theta}_{i}\dottilfbm{\theta}_{i}\\
=&\frac{1}{\beta_{i}}\ttilfbm{\theta}_{i}\left(\bm{\omega}_{i}-\dothatfbm{\theta}_{i}\right)=\ttilfbm{\theta}_{i}\left(\bm{\omega}_{i}-\text{Proj}_{\hfbm{\theta}_{i}}\left(\bm{\omega}_{i}\right)\right)\leq 0\textrm{,}
\end{align*}
with $\bm{\omega}_{i}=-\beta_{i}\tfbm{\Phi}_{i}(\fbm{q}_{i},\dotbm{q}_{i},\fbm{e}_{i},\dotbm{e}_{i})\fbm{s}_{i}$. Using symmetry, the derivative of $V$ can be written:
\begin{align*}
\dot{V}=&-\sum^{N}_{i=1}\left(\tbm{s}_{i}\Sat^{pd}_{i}(\mu_{i}\fbm{s}_{i}+\fbm{e}_{i})-\dotttilbm{q}_{i}\Sat^{pd}_{i}(p_{i}\tilbm{q}_{i})\right)\\
&+\sum^{N}_{i=1}\dothattbm{q}_{i}\Sat^{pd}_{i}(\fbm{e}_{i})-\sum^{N}_{i=1}k_{i}\dothattbm{q}_{i}\dothatbm{q}_{i}\\
&-\sum^{N}_{i=1}\dothattbm{q}_{i}\sum_{j\in\mathcal{N}_{i}(0)}\left(\nabla_{i}\psi(\|\hbm{q}_{ijd}\|)-\nabla_{i}\psi(\|\hbm{q}_{ij}\|)\right)\textrm{.}
\end{align*}

The following lemma facilitates the simplification of $\dot{V}$.
\begin{lemma}\label{lem6}
Let $\Sat(\cdot)$ be a standard saturation function. For any $\fbm{x}\in\mathbb{R}^{n}$, $\fbm{y}\in\mathbb{R}^{n}$, the following inequality holds:
\begin{align*}
-\tbm{x}\Sat(\fbm{x}+\fbm{y})\leq -\tbm{x}\Sat(\fbm{y})\textrm{.}
\end{align*}
\end{lemma}

\begin{proof}
It suffices to show that $-x_{k}\sat(x_{k}+y_{k})\leq -x_{k}\sat(y_{k})$, for any $k=1,\cdots,n$. Let the bound on $\sat(\cdot)$ be $s$.
\begin{itemize}
\item[1. ] $sat(x_{k}+y_{k})=x_{k}+y_{k}$.\\
If $\sign(x_{k})=\sign(y_{k})$, then $-x_{k}\sat(x_{k}+y_{k})=-x^{2}_{k}-x_{k}y_{k}\leq -x_{k}y_{k}$ and $-x_{k}\sat(y_{k})\geq -x_{k}y_{k}$. Hence, $-x_{k}\sat(x_{k}+y_{k})\leq -x_{k}\sat(y_{k})$.\\
If $\sign(x_{k})\neq\sign(y_{k})$, then $-x_{k}\sat(x_{k}+y_{k})=-x^{2}_{k}+|x_{k}y_{k}|$ and $-x_{k}\sat(x_{k}+y_{k})\leq |x_{k}|s$.
\begin{itemize}
\item[\textcircled{1}] If $\sat(y_{k})=y_{k}$, then $-x_{k}\sat(y_{k})=|x_{k}y_{k}|$.
\item[\textcircled{2}] If $\sat(y_{k})=s\cdot\sign(y_{k})$, then $-x_{k}\sat(y_{k})=|x_{k}|s$.
\end{itemize}
Thus, $-x_{k}\sat(x_{k}+y_{k})\leq -x_{k}\sat(y_{k})$ in both cases.

\item[2. ] $sat(x_{k}+y_{k})\neq x_{k}+y_{k}$.\\
If $\sign(x_{k}+y_{k})=\sign(x_{k})$, then $-x_{k}\sat(x_{k}+y_{k})=-|x_{k}|s$ and $-x_{k}\sat(y_{k})\geq -|x_{k}|s$. This gives $-x_{k}\sat(x_{k}+y_{k})\leq -x_{k}\sat(y_{k})$.\\
If $\sign(x_{k}+y_{k})\neq\sign(x_{k})$, then $\sign(y_{k})\neq\sign(x_{k})$ and $|y_{k}|\geq s+|x_{k}|$. It follows that $-x_{k}\sat(x_{k}+y_{k})=|x_{k}|s$ and $-x_{k}\sat(y_{k})=-x_{k}\cdot s\cdot\sign(y_{k})=|x_{k}|s$. Hence, $-x_{k}\sat(x_{k}+y_{k})\leq -x_{k}\sat(y_{k})$.
\end{itemize}
\end{proof}

\lem{lem6} and the definitions of $\fbm{s}_{i}$ and $\fbm{e}_{i}$ lead to:
\begin{align*}
&\dothattbm{q}_{i}\Sat^{pd}_{i}(\fbm{e}_{i})-\tbm{s}_{i}\Sat^{pd}_{i}(\mu_{i}\fbm{s}_{i}+\fbm{e}_{i})+\dotttilbm{q}_{i}\Sat^{pd}_{i}(p_{i}\tilbm{q}_{i})\\
\leq&\left(\dothatbm{q}_{i}-\fbm{s}_{i}+\dottilbm{q}_{i}\right)^\mathsf{T}\Sat^{pd}_{i}(\fbm{e}_{i})=-\alpha\tbm{e}_{i}\Sat^{pd}_{i}(\fbm{e}_{i})\textrm{,}
\end{align*}
which, in turn, yields the following upper-bounding of $\dot{V}$:
\begin{equation}\label{equ35}
\begin{aligned}
\dot{V}\leq& -\sum^{N}_{i=1}\dothattbm{q}_{i}\sum_{j\in\mathcal{N}_{i}(0)}\left(\nabla_{i}\psi(\|\hbm{q}_{ijd}\|)-\nabla_{i}\psi(\|\hbm{q}_{ij}\|)\right)\\
& -\alpha\sum^{N}_{i=1}\tbm{e}_{i}\Sat^{pd}_{i}(\fbm{e}_{i})-\sum^{N}_{i=1}k_{i}\dothattbm{q}_{i}\dothatbm{q}_{i}\textrm{.}
\end{aligned}
\end{equation}

Another assumption used to investigate the connectivity-preserving consensus of uncertain Euler-Lagrange MAS-s that start from rest and have bounded actuation and time-varying communication delays is the following.
\begin{assumption}\label{ass6}
Let the potential function~\eqref{equ1} have a fourth property: if $|\hbm{q}^{k}_{ij}|$ and $|\hbm{q}^{k}_{ijd}|$ are bounded for all $k=1,\cdots,n$, then there exist $\zeta>0$ and $\vartheta>0$ such that
\begin{align*}
\big|\nabla^{k}_{i}\psi(\|\hbm{q}_{ij}\|)-\nabla^{k}_{i}\psi(\|\hbm{q}_{ijd}\|)\big|\leq\zeta |\hat{q}^{k}_{j}-\hat{q}^{k}_{jd}|+\vartheta\bar{q}_{j}
\end{align*}
with $\bar{q}_{j}=\|\hbm{q}_{j}-\hbm{q}_{jd}\|_{\infty}$.
\end{assumption}
\rem{rem11}, proven in~\cite{Mine2018}, shows how to satisfy \ass{ass6}. 

The treatment of time-varying delays is also facilitated by the following lemma.
\begin{lemma}\label{lem7}\cite{Mine2018}
Given $\dothatbm{q}_{i}$, $\overline{q}_{j}$ in~\ass{ass6}, and any variable time delays $T(t)$ bounded by $0\leq T(t)\leq \overline{T}$, they obey
\begin{align*}
&\int^{t}_{0}\dothattbm{q}_{i}(\sigma)\overline{q}_{j}(\sigma)\fbm{1}d\sigma\\
&\leq\frac{a}{2}\int^{t}_{0}\|\dothatbm{q}_{i}(\sigma)\|^{2}d\sigma +\frac{n\overline{T}^{2}}{2a}\int^{t}_{0}\|\dothatbm{q}_{j}(\sigma)\|^{2}d\sigma
\end{align*}
for any $a>0$.
\end{lemma}
\ass{ass6}, Lemma 1 in~\cite{Nuno2009} and~\lem{lem7} lead to:
\begin{align*}
&-\int^{t}_{0}\dothattbm{q}_{i}(\sigma)\left[\nabla_{i}\psi(\|\hbm{q}_{ijd}(\sigma)\|)-\nabla_{i}\psi(\|\hbm{q}_{ij}(\sigma)\|)\right]d\sigma\\
&\leq \zeta\int^{t}_{0}|\dothatbm{q}_{i}(\sigma)|^\mathsf{T}\int^{\sigma}_{\sigma-T_{ji}(\sigma)}|\dothatbm{q}_{j}(\theta)|d\theta d\sigma + \vartheta\int^{t}_{0}\dothattbm{q}_{i}(\sigma)\overline{q}_{j}(\sigma)\fbm{1}d\sigma\\
&\leq \frac{a(\zeta+\vartheta)}{2}\int^{t}_{0}\|\dothatbm{q}_{i}(\sigma)\|^{2}d\sigma+\frac{(\zeta+n\vartheta)\overline{T}^{2}_{ji}}{2a}\int^{t}_{0}\|\dothatbm{q}_{j}(\sigma)\|^{2}d\sigma\textrm{,}
\end{align*}
where $\overline{T}_{ji}$ is the upper bound of $T_{ji}(t)$. Further, time integration of~\eqref{equ35} yields:
\begin{align*}
&V(t)\leq V(0)-\sum^{N}_{i=1}\int^{t}_{0}\alpha\tbm{e}_{i}(\sigma)\Sat^{pd}_{i}\left[\fbm{e}_{i}(\sigma)\right]+k_{i}\|\dothatbm{q}_{i}(\sigma)\|^{2}d\sigma\\
&+\sum^{N}_{i=1}\sum_{j\in\mathcal{N}_{i}(0)}\left[\frac{a(\zeta+\vartheta)}{2}+\frac{(\zeta+n\vartheta)\overline{T}^{2}_{ij}}{2a}\right]\int^{t}_{0}\|\dothatbm{q}_{i}(\sigma)\|^{2}d\sigma\textrm{.}
\end{align*}
If the damping injection $k_{i}$ in the proxies is sufficiently large:
\begin{align*}
k_{i}\geq\sum_{j\in\mathcal{N}_{i}(0)}\left(\frac{a(\zeta+\vartheta)}{2}+\frac{(\zeta+n\vartheta)\overline{T}^{2}_{ij}}{2a}\right)+\upsilon_{i}
\end{align*}
with $\upsilon_{i}>0$ and $i=1\cdots,N$, it follows that:
\begin{align*}
V(t)\leq V(0)-\sum^{N}_{i=1}\int^{t}_{0}\alpha\tbm{e}_{i}(\sigma)\Sat^{pd}_{i}\left[\fbm{e}_{i}(\sigma)\right]+\upsilon_{i}\|\dothatbm{q}_{i}(\sigma)\|^{2}d\sigma\textrm{.}
\end{align*}

After selecting $\hbm{q}_{i}(0)=\fbm{q}_{i}(0)$, $\dothatbm{q}_{i}(0)=\fbm{0}$ and $\underline{{\theta}}^{k}_{i}\leq {\theta}^{k}_{i}(0)\leq\overline{{\theta}}^{k}_{i}$ for $k=1,\cdots,n$, \ass{ass4} leads to $\dotbm{q}_{i}=\fbm{0}$ and, further, to $\fbm{s}_{i}=\fbm{0}$, $\tilbm{q}_{i}(0)=\fbm{0}$ and $\Phi_{i}(\tilbm{q}_{i}(0))=0$. It then follows that: 
\begin{align*}
V(0)=\sum^{N}_{i=1}\frac{1}{2\beta_{i}}\ttilfbm{\theta}_{i}(0)\tilfbm{\theta}_{i}(0)+\Psi(\hbm{q}(0))\textrm{.}
\end{align*}
As in~\sect{sec: saturations}, let $\phi^{*}=\min \limits_{i=1\cdots,N}(\phi^{*}_{i})$ with $\phi^{*}_{i}$ the minimum of $\phi_{i}(\tilbm{q}_{i})$ on the boundary of $B(\fbm{0},\frac{\epsilon}{3})$. Further, choose a bounded potential function $\Psi(\hbm{q})$ that: 1) attains its maximum $\Psi_{max}=\phi^{*}$ whenever there is an edge $(i,j)\in\mathcal{E}(0)$ such that $\|\hbm{q}_{ij}\|=\hat{r}=r-\frac{2\epsilon}{3}$; 2) and obeys $\Psi(\hbm{q}(0))<\Psi_{max}$. There then exists $\beta_{i}>0$ such that:
\begin{align*}
\sum^{N}_{i=1}\frac{1}{2\beta_{i}}\ttilfbm{\theta}_{i}(0)\tilfbm{\theta}_{i}(0)<\delta=\Psi_{max}-\Psi(\hbm{q}(0))\textrm{,}
\end{align*} 
which, in turn, leads to $V(t)\leq V(0)<\phi^{*}$ and, further, to $\phi_{i}(\tilbm{q}_{i}(t))<\phi^{*}$ and to $\Psi(\hbm{q}(t))<\phi^{*}$. Then, it follows that $\|\hbm{q}_{ij}(t)\|<\hat{r}$ for each $(i,j)\in\mathcal{E}(0)$, and that $\|\tilbm{q}_{i}(t)\|<\frac{\epsilon}{3}$ for each $i=1,\cdots,N$, which together imply that $\|\fbm{q}_{ij}(t)\|\leq\|\tilbm{q}_{i}(t)\|+\|\hbm{q}_{ij}(t)\|+\|\tilbm{q}_{j}(t)\|<r$ for every link $(i,j)\in\mathcal{E}(0)$, i.e., the initial connectivity of the MAS is maintained.

Time integration of $\dot{V}$ implies that $\fbm{e}_{i}\in\mathcal{L}_{2}\cap\mathcal{L}_{\infty}$ and $\dothatbm{q}_{i}\in\mathcal{L}_{2}\cap\mathcal{L}_{\infty}$, $i=1,\cdots,N$, which, in turn, imply that $\fbm{e}_{i}=\fbm{q}_{i}-\hbm{q}_{i}\to\fbm{0}$ and that $\dothatbm{q}_{i}\to\fbm{0}$. From~\eqref{equ32}, it follows that $\dddothbm{q}_{i}\in\mathcal{L}_{\infty}$ and, further, that $\sum_{j\in\mathcal{N}_{i}(0)}\nabla_{i}\psi(\|\hbm{q}_{ijd}\|)\to\sum_{j\in\mathcal{N}_{i}(0)}\nabla_{i}\psi(\|\hbm{q}_{ij}\|)\to\fbm{0}$ because $\dotbm{q}_{j}\to\fbm{0}$. As in~\sect{sec: single-integrator}, this means that $\hbm{q}_{i}-\hbm{q}_{j}\to\fbm{0}$, i.e., all agents converge to the same configuration and achieve consensus.

\begin{remark}
\normalfont In~\eqref{equ32}, $\hbm{f}_{i}$ includes: a dynamic compensation term $\bm{\Phi}_{i}(\fbm{q}_{i},\dotbm{q}_{i},\fbm{e}_{i},\dotbm{e}_{i})\hfbm{\theta}_{i}$ similar to the dynamic compensation term in~\eqref{equ23}; and a saturated term $\Sat^{pd}_{i}(\mu_{i}\fbm{s}_{i}+\fbm{e}_{i})$ which is a bounded version of $\fbm{e}_{i}$ that eliminates the need to analyze saturated agent-proxy couplings. All analysis in this section holds even without including $\mu_{i}\fbm{s}_{i}$ in $\Sat^{pd}_{i}(\cdot)$ in~\eqref{equ32}. However, rewriting $\mu_{i}\fbm{s}_{i}+\fbm{e}_{i}$ in the P+d form $\mu_{i}\dotbm{q}_{i}+(\alpha\mu_{i}+1)\fbm{e}_{i}$ illustrates that $\mu_{i}\fbm{s}_{i}$ provides flexibility in tuning the damping injection in the designed agent controls $\hbm{f}_{i}$.
\end{remark}

\section{Simulations}\label{sec: simulations}
In this section, simulations of kinematic and EL networks adopt the potential function in~\eqref{equ3} and in~\rem{rem11}, respectively. 

\subsection{Kinematic MAS-s}
This section starts with a comparison of the controllers in~\eqref{equ6} and~\cite{Khorasani2013ACC} through simulations of a single-integrator MAS with $N=5$ agents with: dimension $n=1$; communication radius $r=1$~m; actuation bounds $2$~m/s, $3$~m/s, $1$~m/s, $2$~m/s and $3$~m/s; and initial positions $1$~m, $1.5$~m, $2$~m, $2.5$~m and $3$~m. For the controller in~\eqref{equ6}, the third property of~$\Psi(\fbm{q})$ is guaranteed by choosing $Q=0.025$ after selecting $\epsilon=0.1$~m. For the controller in~\cite{Khorasani2013ACC}, the design parameters are $\alpha_{i}=1$, $\epsilon_{i}=0.1$ and $\beta_{i}=11$.

\figsa{fig11a}{fig11b} illustrate that the MAS is coordinated in about $0.7$~s by the controller in~\eqref{equ6}, and in about $10$~s by the controller in~\cite{Khorasani2013ACC}. The difference in the coordination performance arises from a dynamic compensation mechanism in~\cite{Khorasani2013ACC} that prevents actuator saturation, as seen in \fig{fig11d}, and thus makes the control conservative and limits convergence speed. In contrast, the controller~\eqref{equ6} permits simultaneous saturation of several actuators, as shown in~\fig{fig11c}, and more fully uses the actuation of the agents.
\begin{figure}[!hbt]
\centering
\subfigure[Coordination of a $5$ single-integrator MAS controlled by~\eqref{equ6}.]{
	\label{fig11a}
	\includegraphics[width=4cm,height=2.5cm]{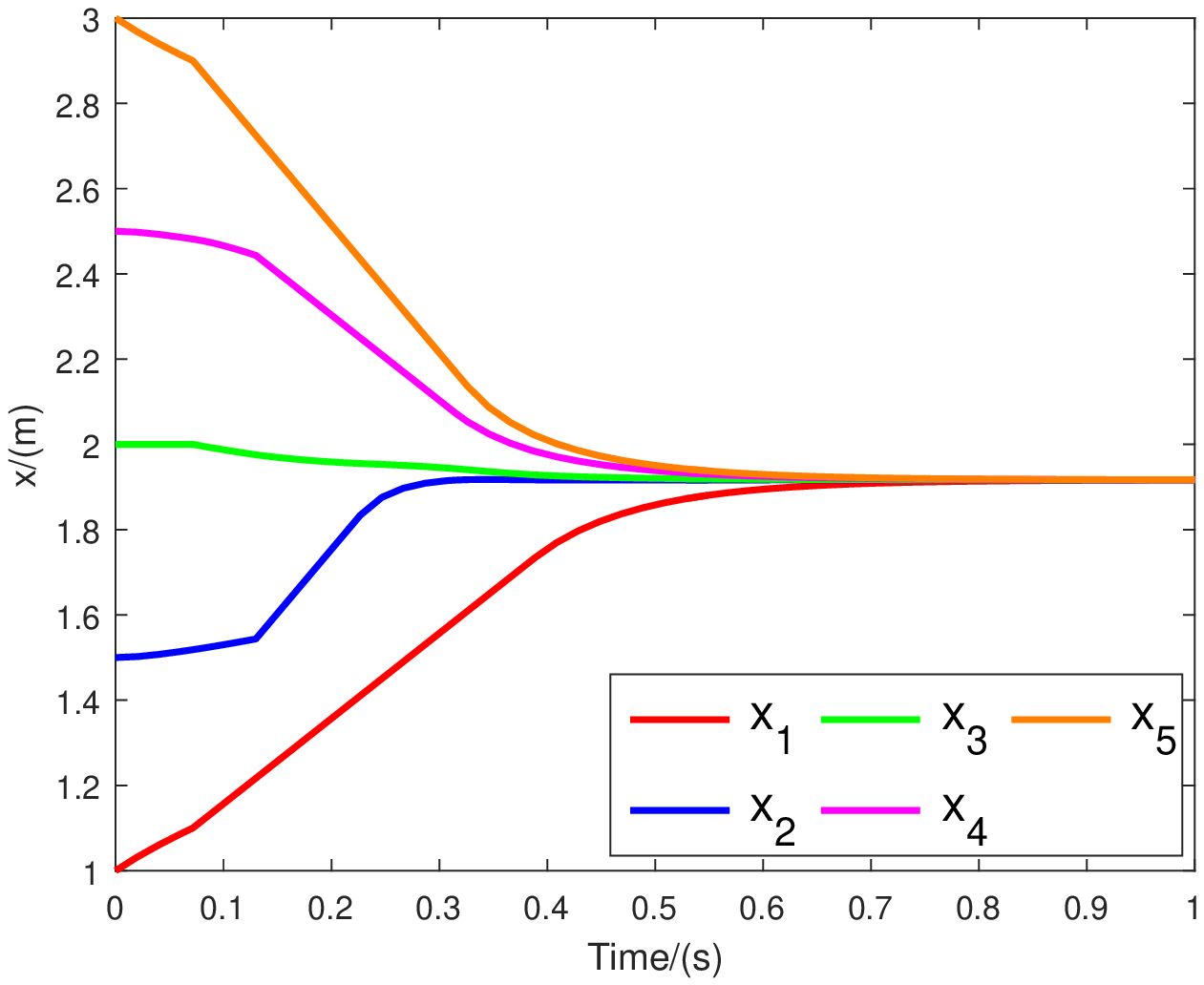}}\quad
\subfigure[Coordination of a $5$ single-integrator MAS controlled by~\cite{Khorasani2013ACC}.]{
	\label{fig11b}
	\includegraphics[width=4cm,height=2.5cm]{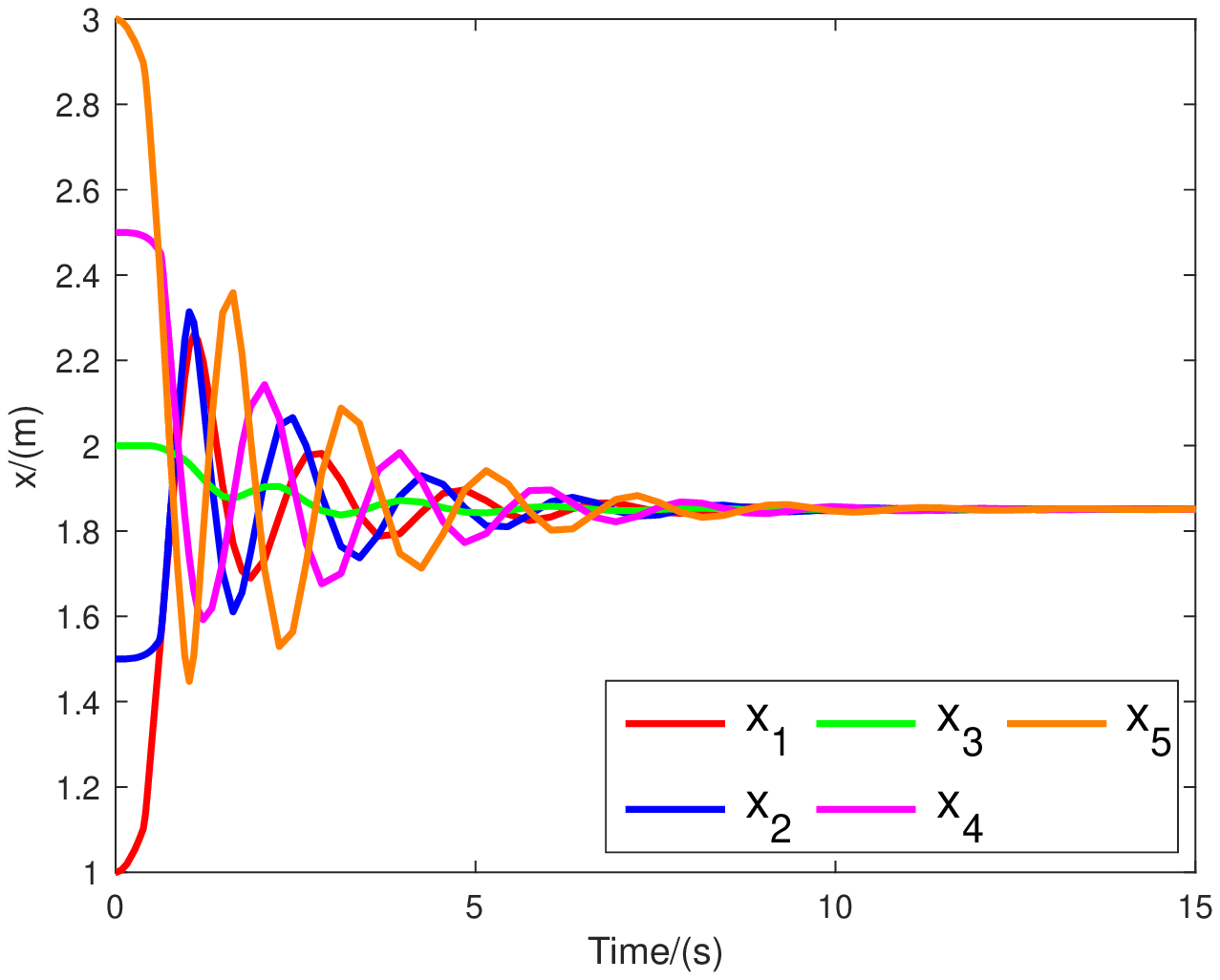}}	
	
\subfigure[Saturated actuations of a $5$ single-integrator MAS controlled by~\eqref{equ6}.]{
	\label{fig11c}
	\includegraphics[width=4cm,height=2.5cm]{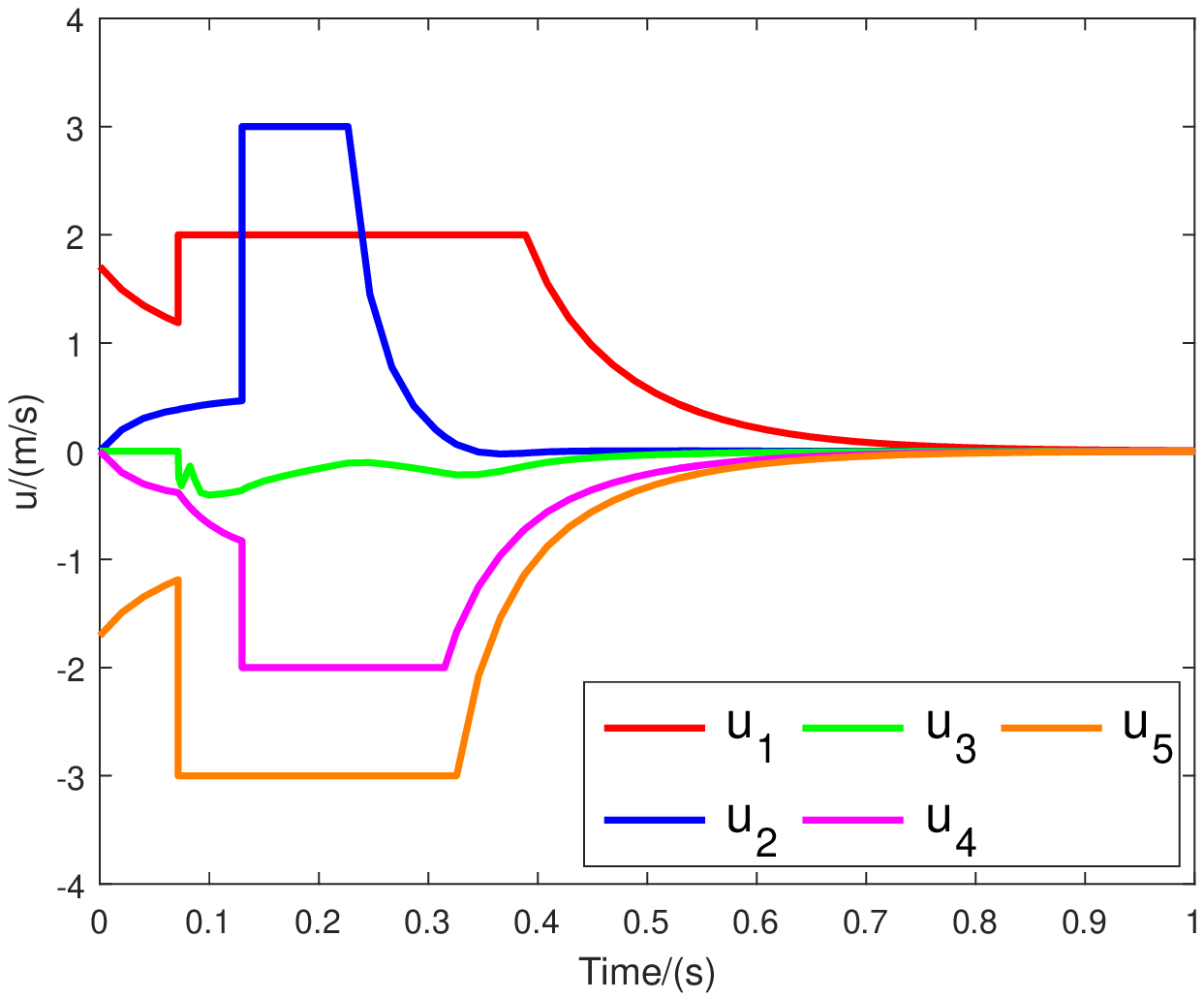}}\quad
\subfigure[Saturated actuations of a $5$ single-integrator MAS controlled by~\cite{Khorasani2013ACC}.]{
	\label{fig11d}
	\includegraphics[width=4cm,height=2.5cm]{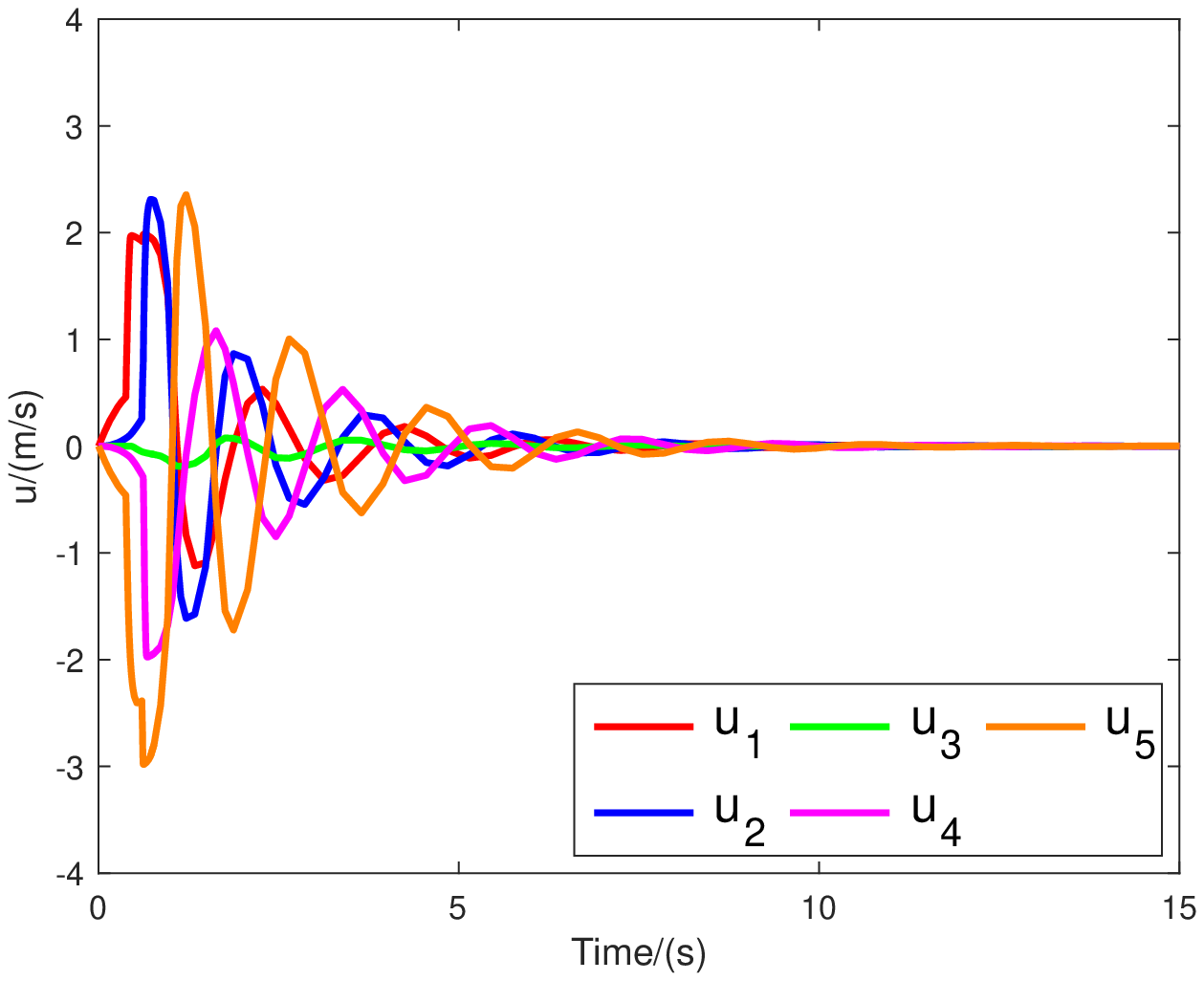}}
\caption{Comparisons of coordination speeds~(in (a) and (b)) and actuations~(in (c) and (d)) between the controller~\eqref{equ6} and the controller in~\cite{Khorasani2013ACC}.}
\label{fig11}
\end{figure}

A second simulation depicts a MAS with $5$ nonholonomic agents under the control~\eqref{equ12}.  The agents have: communication radius $r=1$~m; translation actuation bounds $\overline{s}_{vi}$ equal to $0.5$~m/s, $0.3$~m/s, $0.4$~m/s, $0.2$~m/s and $0.4$~m/s; orientation actuation bounds $\overline{s}_{\omega i}$ equal to $2$~rad/s, $3$~rad/s, $1$~rad/s, $2$~rad/s and $3$~rad/s; and initial configurations $[0.5\text{~m}, 0.5\text{~m}, -0.4\pi]^\mathsf{T}$, $[1\text{~m}, 0\text{~m}, 0.7\pi]^\mathsf{T}$, $[2\text{~m}, 1\text{~m}, 0.5\pi]^\mathsf{T}$, $[0.5\text{~m}, 1\text{~m}, -0.6\pi]^\mathsf{T}$ and $[1.5\text{~m}, 0.5\text{~m}, 0.1\pi]^\mathsf{T}$. After selecting $\epsilon=0.1$~m, the third property of~$\Psi(\fbm{q})$ is guaranteed by choosing $Q=0.025$. The gain of the orientation controller in~\eqref{equ12} is chosen $k=2$ heuristically. The trajectories and translation actuation signals of all agents are depicted in~\figsa{fig12a}{fig12b}, respectively. These figures verify that the controller~\eqref{equ12} coordinates the nonholonomic MAS although all agent actuators saturate during various time periods.
\begin{figure}[!hbt]
\centering
\subfigure[Position coordination of $5$ nonholonomic agents controlled by~\eqref{equ12}. The black lines connect the initially adjacent agents.]{
	\label{fig12a}
	\includegraphics[width=4cm,height=3cm]{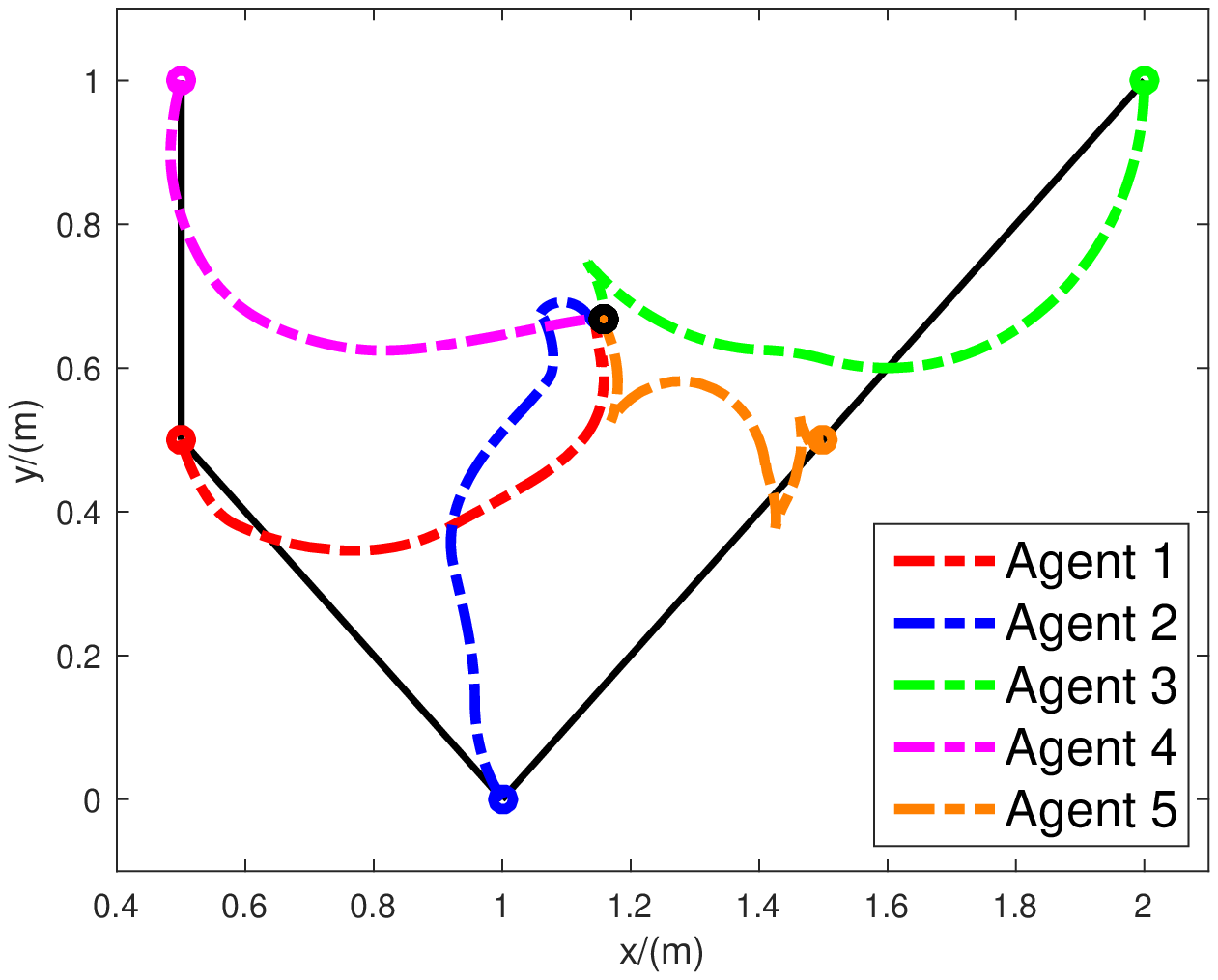}}\quad 
\subfigure[Saturated translation actuations of $5$ nonholonomic agents controlled by~\eqref{equ12}.]{
	\label{fig12b}
	\includegraphics[width=4cm,height=3cm]{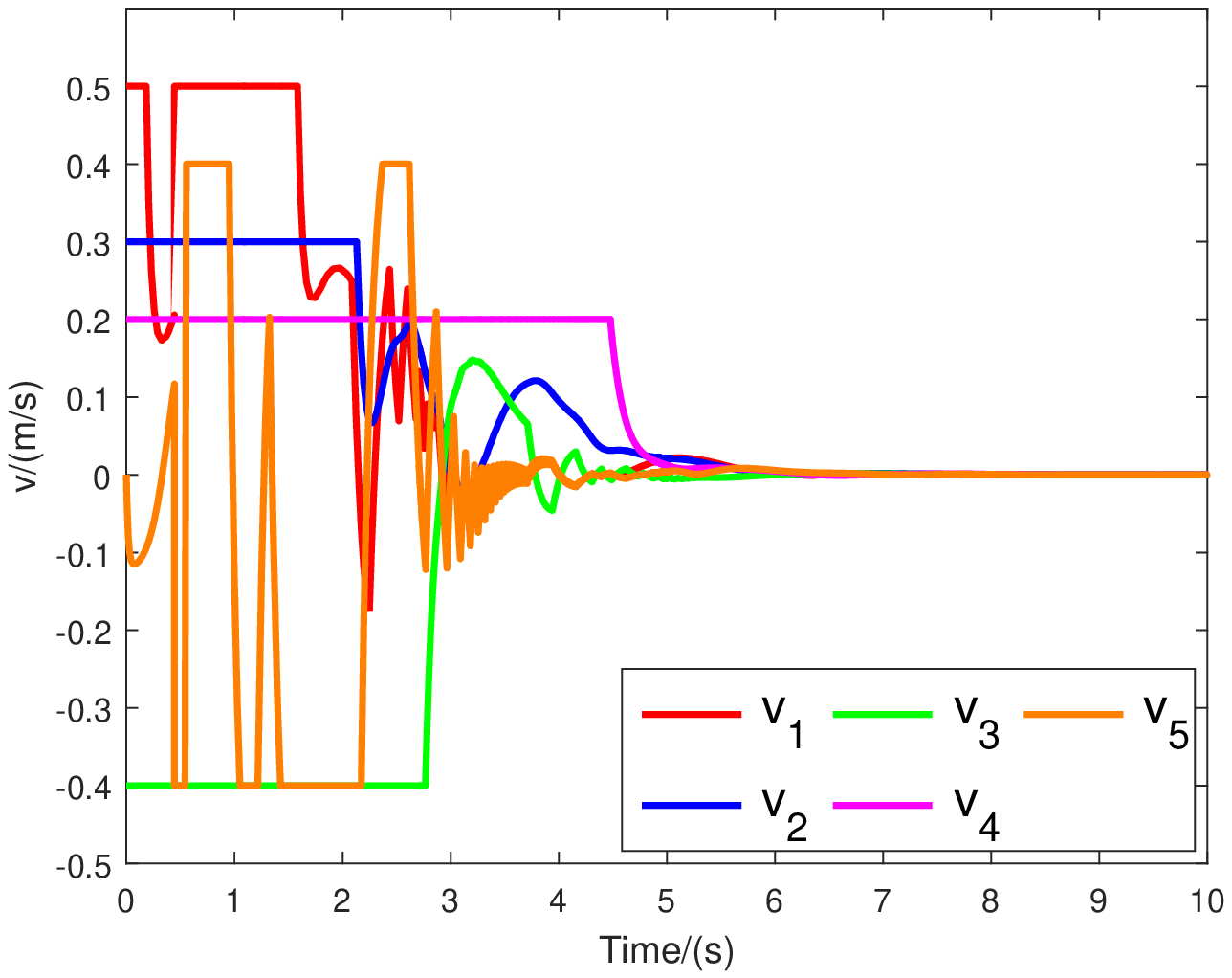}}
\caption{Position synchronization of a MAS with $5$ nonholonomic agents under the controller~\eqref{equ12}, regardless of actuator saturation.}
\label{fig12}
\end{figure}

\subsection{Euler-Lagrange MAS-s}
This section validates that the controllers designed in~\sect{sec: EL} preserve the local connectivity of a simulated MAS with $N=5$ robots despite bounded actuation. Each robot is a $2$-degree-of-freedom~($2$-DOF) manipulator with link masses $m_{k}=0.5$~kg and lengths $l_{k}=1$~m. In task space, the communication radius of the end effector of each robot is $r=1$~m. The robots are initially at rest at $\fbm{q}_{1}=[\pi/12, -5\pi/12]^\mathsf{T}$, $\fbm{q}_{2}=[\pi/6, -\pi/3]^\mathsf{T}$, $\fbm{q}_{3}=[\pi/4, -\pi/4]^\mathsf{T}$, $\fbm{q}_{4}=[\pi/3, -\pi/4]^\mathsf{T}$ and $\fbm{q}_{5}=[5\pi/12, -5\pi/12]^\mathsf{T}$. Selecting $\epsilon=0.25$~m guarantees \ass{ass2}. To preserve connectivity and coordinate the end effectors, the robot controllers are designed in task space.    

For the output feedback control~\eqref{equ19}, $\dothatbm{q}_{i}(0)=\fbm{0}$ can be guaranteed by choosing $\hbm{q}_{i}(0)=\kappa_{i}\fbm{q}_{i}(0)$. Letting $Q=0.2$ ensures $V(0)\leq\Psi_{max}$ in~\eqref{equ20}. Then, $s_{i}=5$ and $\kappa_{i}=12$ are selected heuristically. The positions of the $5$ end effectors along the $x$- and $y$-axes are depicte in~\figsa{fig13a}{fig13b}, respectively. The convergence of the paths of the $5$ end effectors is shown in~\fig{fig14a}. Velocity estimation causes the end effectors to twist and turn during coordination. Increased damping in the proxies can smooth the end effector paths at the expense of convergence speed.
\begin{figure}[!hbt]
\centering
\subfigure[End effector positions in task space along the $x$-axis.]{
	\label{fig13a}
	\includegraphics[width=4cm,height=2.5cm]{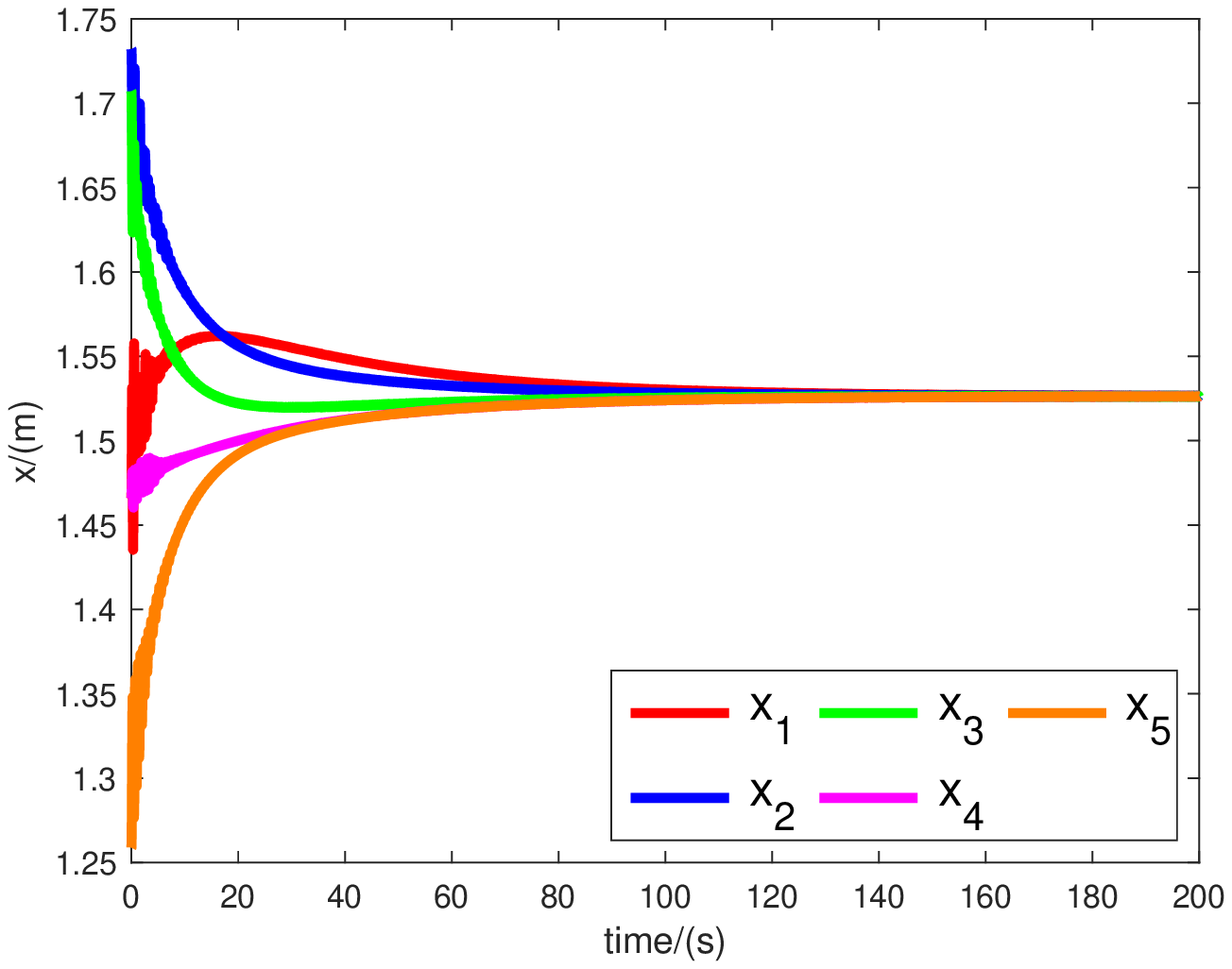}}\quad
\subfigure[End effector positions in task space along the $y$-axis.]{
	\label{fig13b}
	\includegraphics[width=4cm,height=2.5cm]{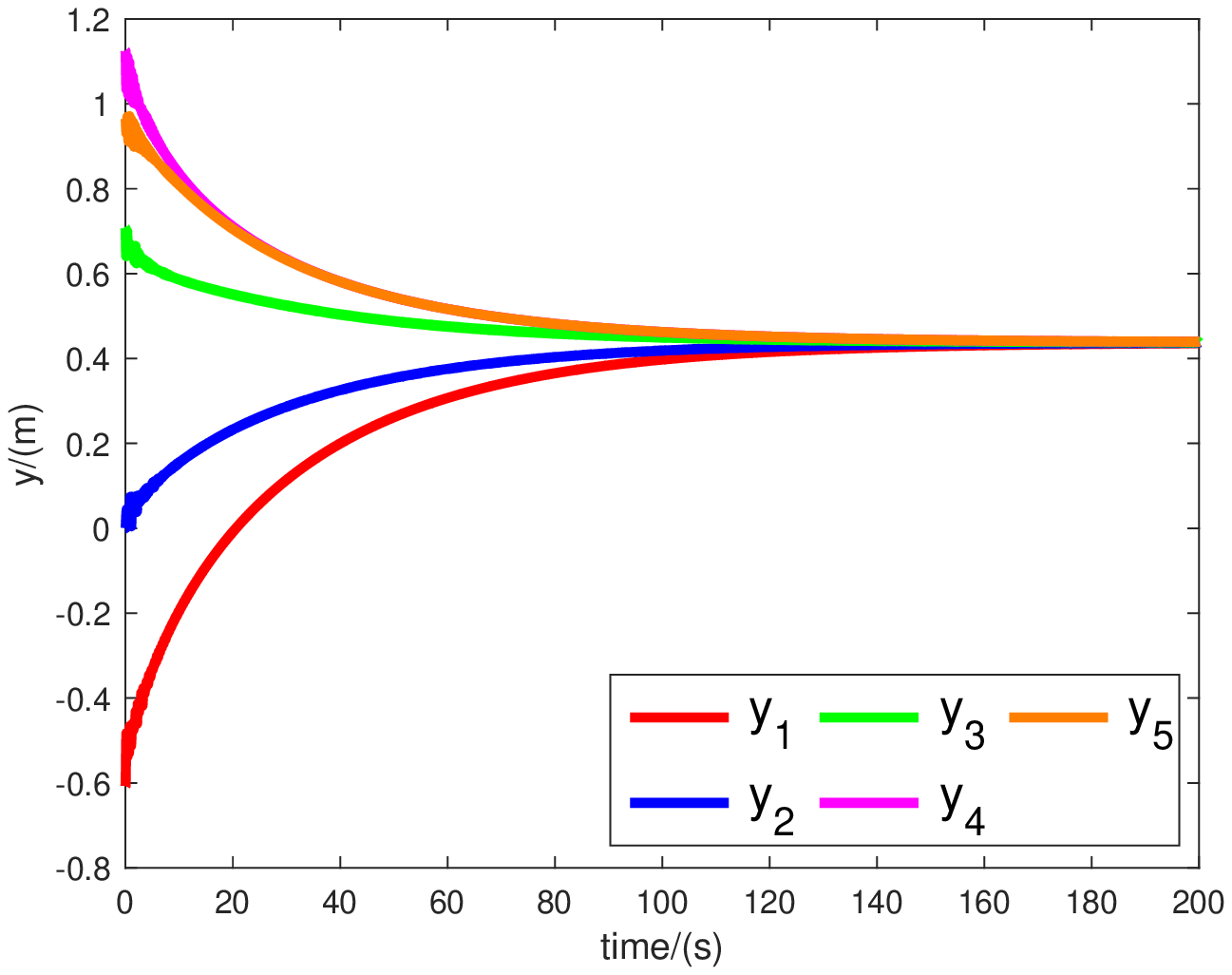}}
\caption{End effector positions of a $5$-robots MAS under the output feedback controller~\eqref{equ19}.}
\label{fig13}
\end{figure}
\begin{figure}[!hbt]
\centering
\subfigure[End effector task space paths under output feedback control.]{
	\label{fig14a}
	\includegraphics[width=4cm,height=3cm]{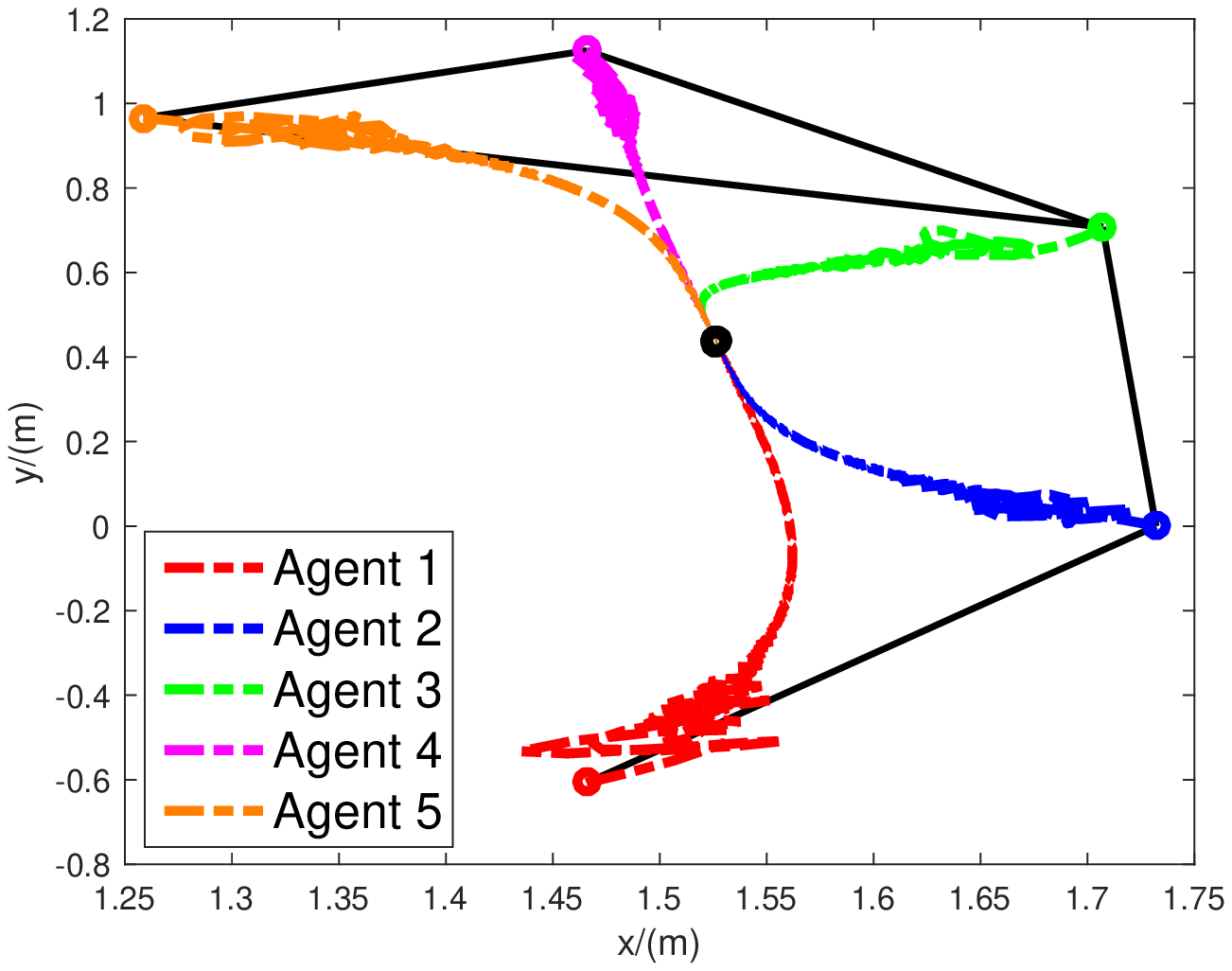}}\quad
\subfigure[End effector task space paths under adaptive control.]{
	\label{fig14b}
	\includegraphics[width=4cm,height=3cm]{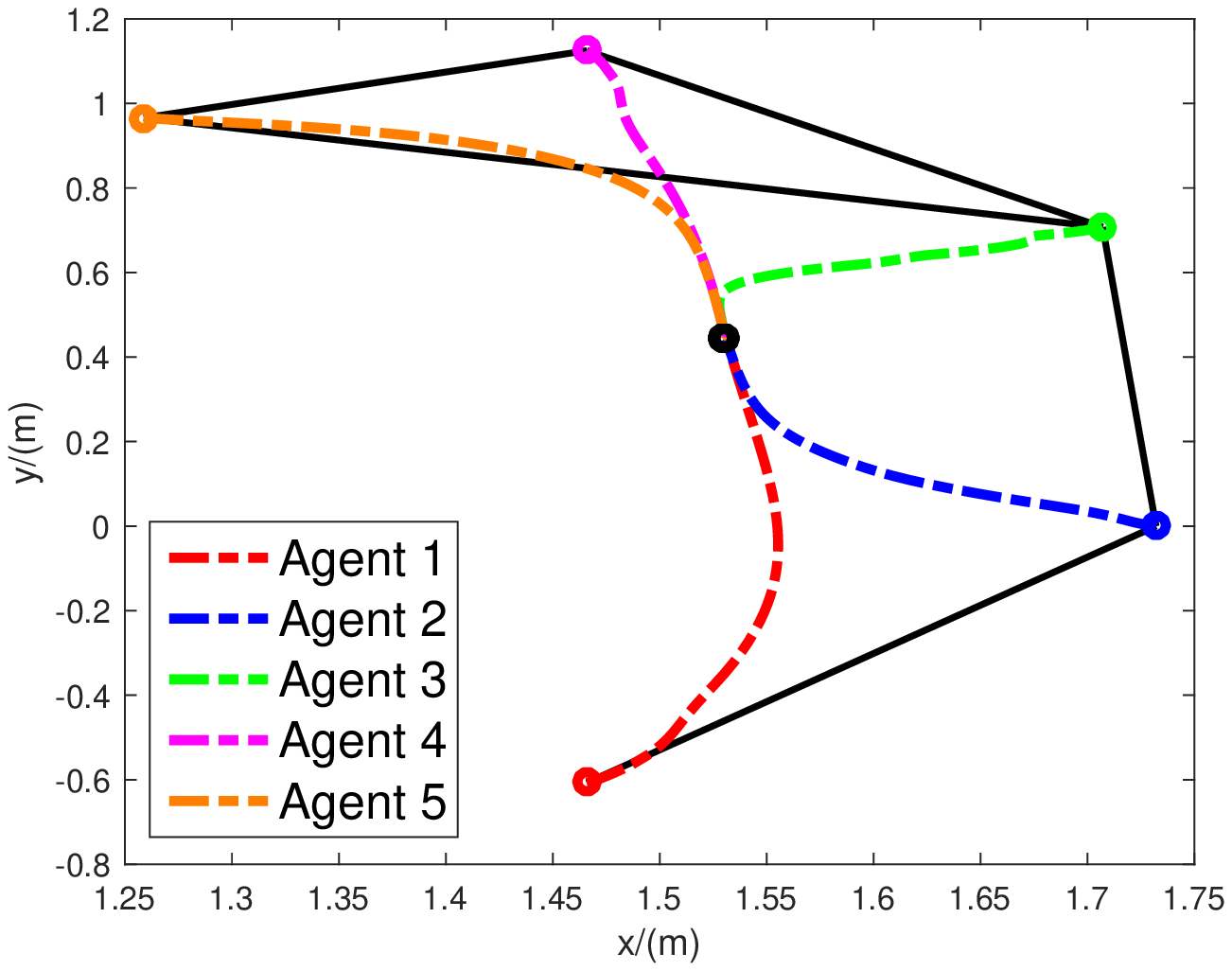}}
\caption{Consensus of the end effectors of a $5$-robots MAS under output feedback control~\eqref{equ19} and under adaptive control~\eqref{equ23}. Black lines connect the initially adjacent end-effectors.}
\label{fig14}
\end{figure}

Assuming that the parameters of the $5$-robots MAS are unknown, the adaptive controller~\eqref{equ23} can be designed by: 1) selecting $Q=0.1$ such that $\Psi(\fbm{q}(0))<\Psi_{max}$; 2) setting $\alpha=1$ and $\kappa=50$ heuristically; and 3) letting $\mu_{i}=5$ be sufficiently large. Then, $V(0)\leq\Psi_{max}$ is guaranteed. The convergence of the paths of all end effectors to the same point is depicted in~\fig{fig14b}. Detailed position information along $x$- and $y$-axes is displayed in~\fig{fig15}.
\begin{figure}[!hbt]
\centering
\subfigure[End effector positions in task space along $x$-axis.]{
	\label{fig15a}
	\includegraphics[width=4cm,height=2.5cm]{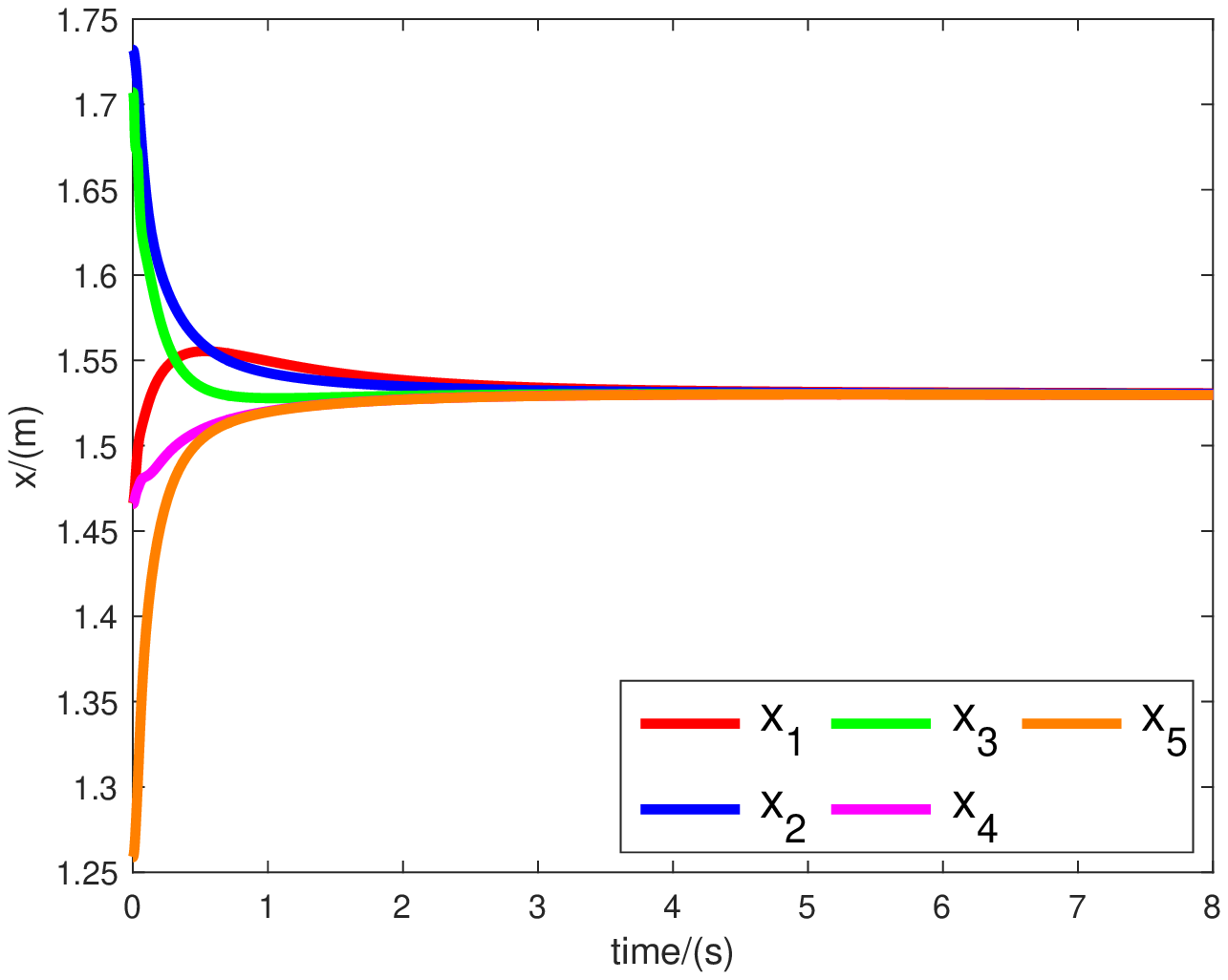}}
\subfigure[End effector positions in task space along $y$-axis.]{
	\label{fig15b}
	\includegraphics[width=4cm,height=2.5cm]{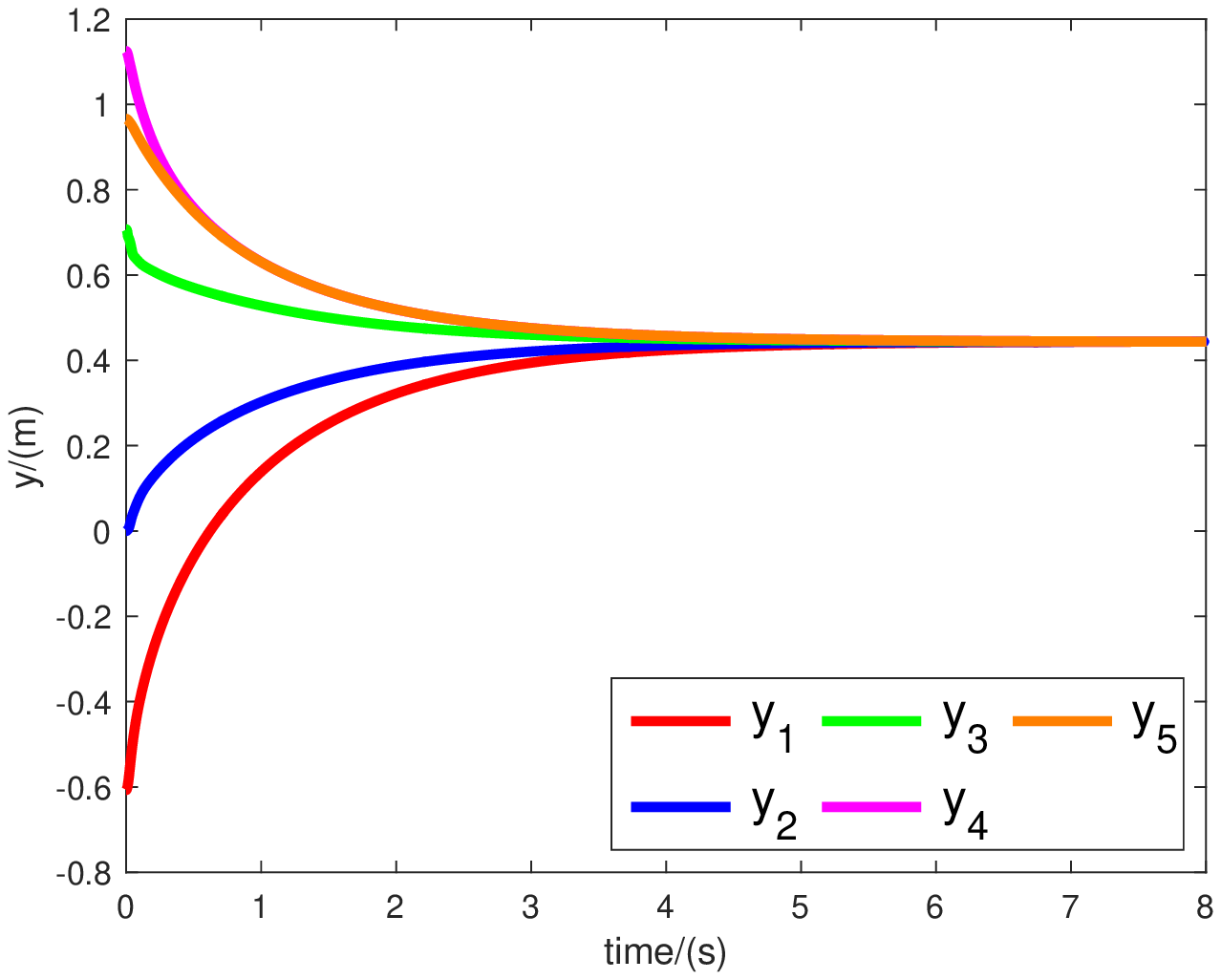}}
\caption{End effector positions of the $5$-robots MAS under the adaptive controller~\eqref{equ23}.}
\label{fig15}
\end{figure}

Assuming that, for all $5$ robots, the actuation is bounded by $\bar{\fbm{f}}_{i}=[30,10]^\mathsf{T}$~N, and the gravity terms are bounded by $\bm{\gamma}_{i}=[15,5]^\mathsf{T}$~N, the standard saturations in~\eqref{equ27} then have bounds $\bar{\fbm{f}}_{i}-\bm{\gamma}_{i}$. Selecting $k_{i}=3$ and $p_{i}=100$ leads to $\phi^{*}=0.25$. Letting $Q=0.04$ and $\rho=0.01$ guarantees that $V(0)<\Psi_{max}$ and $\Psi_{max}\leq\phi^{*}$. \figsa{fig16a}{fig16b} show that~\eqref{equ27} drives all agents to the same configuration. Because proxy damping gains are small, the agents move away from their final configuration initially. However, the MAS reaches consensus without breaking any edge eventually, see~\fig{fig17a}. 
\begin{figure}[!hbt]
\centering
\subfigure[End effector positions in task space along $x$-axis.]{
	\label{fig16a}
	\includegraphics[width=4cm,height=2.5cm]{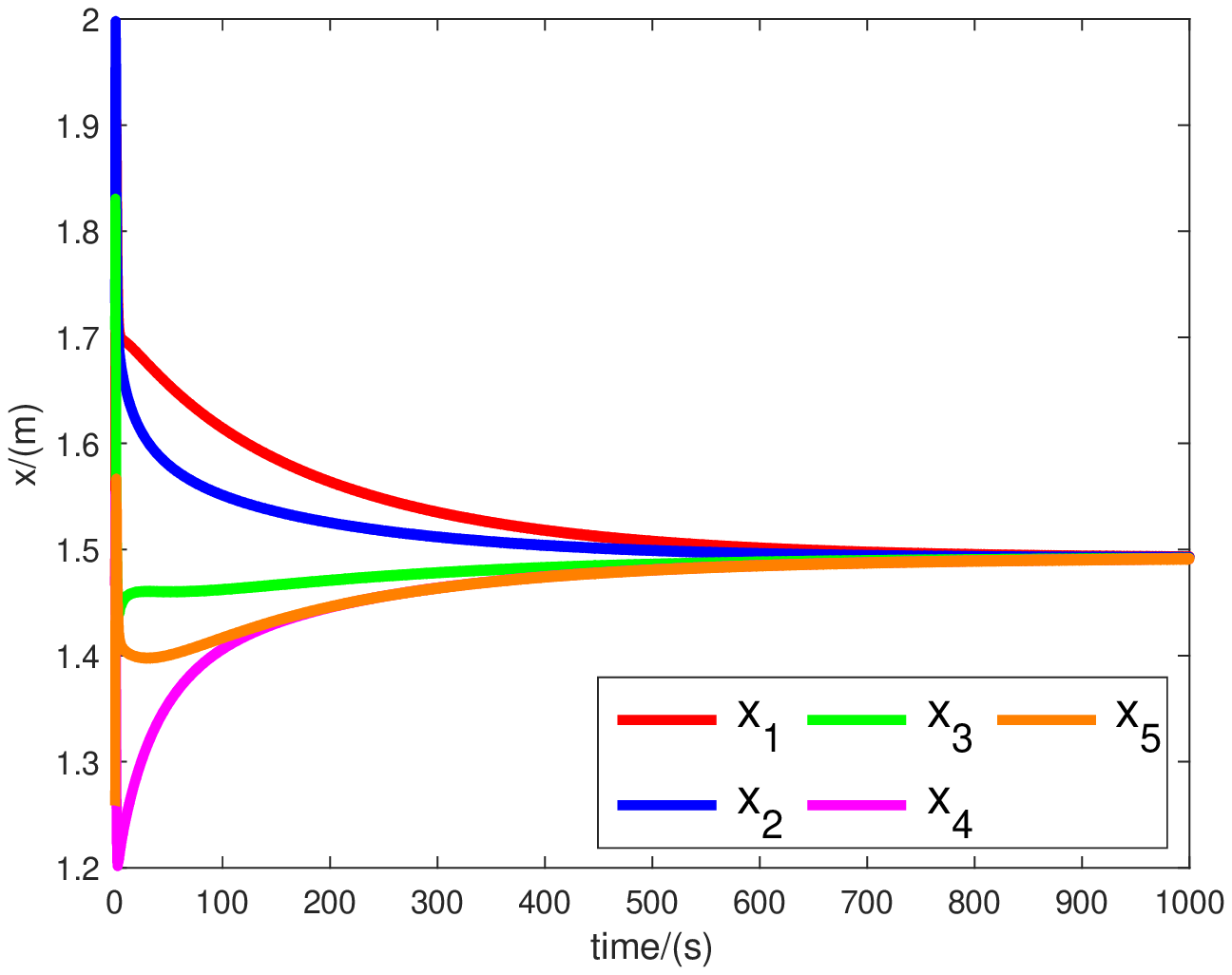}}
\subfigure[End effector positions in task space along $y$-axis.]{
	\label{fig16b}
	\includegraphics[width=4cm,height=2.5cm]{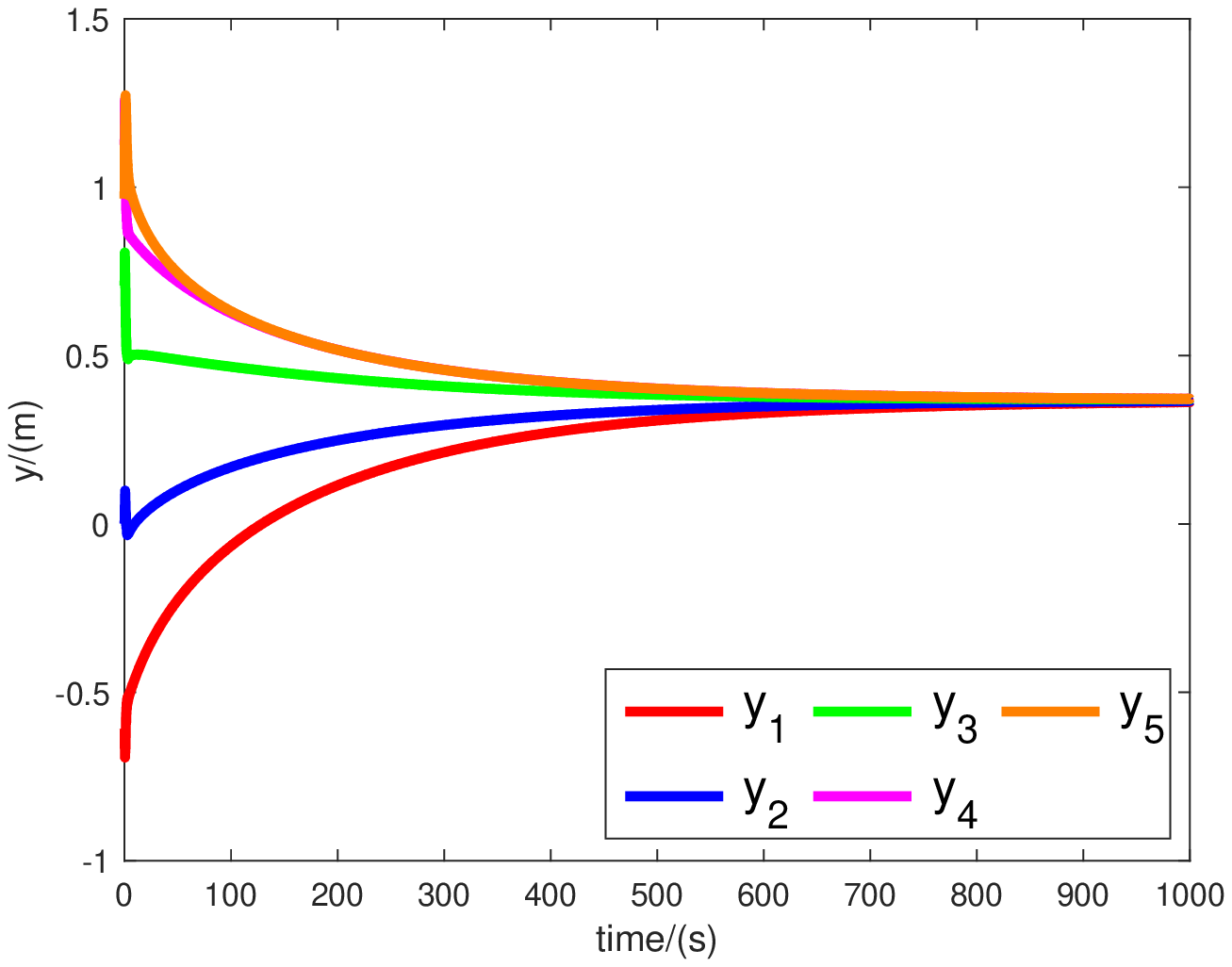}}
\caption{End effector positions of the $5$-robots MAS with bounded actuation under the control~\eqref{equ27}.}
\label{fig16}
\end{figure}

\begin{figure}[!hbt]
\centering
\subfigure[Coordination by~\eqref{equ27}.]{
	\label{fig17a}
	\includegraphics[width=4cm,height=3cm]{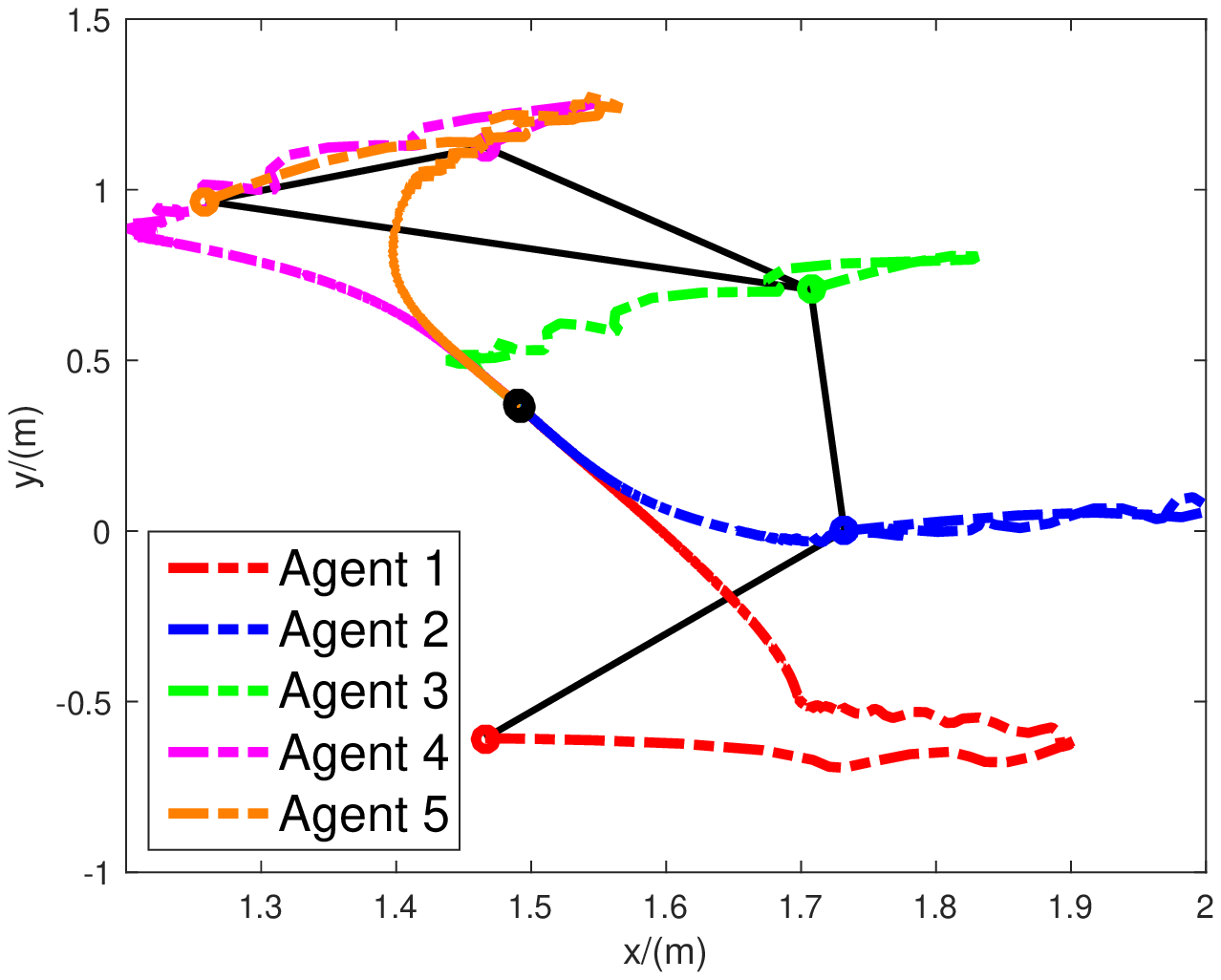}}
\subfigure[Coordination by~\eqref{equ32}.]{
	\label{fig17b}
	\includegraphics[width=4cm,height=3cm]{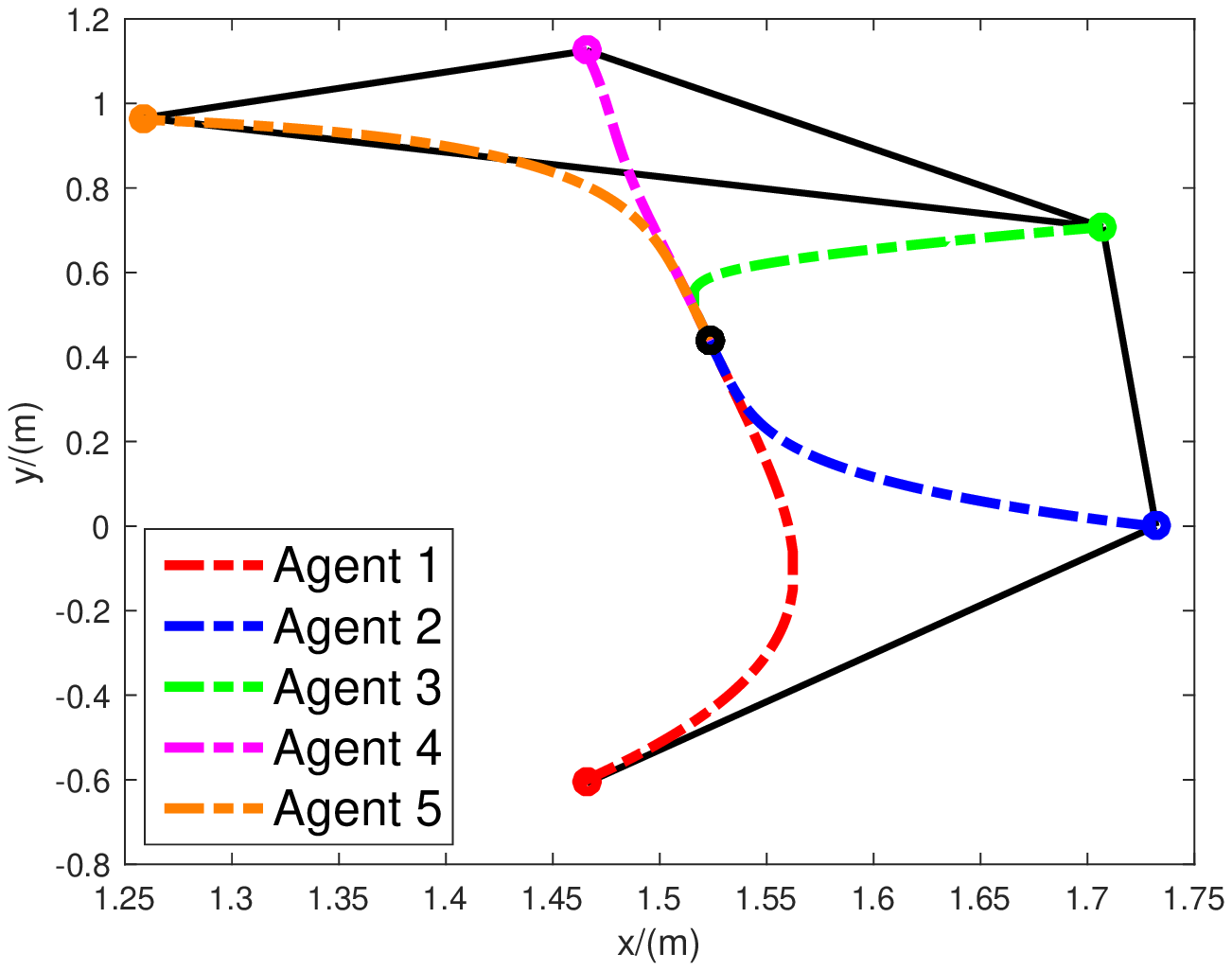}}
\caption{Coordination in task space of the $5$-robots MAS with limited actuation and controlled by~\eqref{equ27} and by~\eqref{equ32}.}
\label{fig17}
\end{figure}

In the last set of simulations, the time-varying communication delays are bounded by $\overline{T}_{ji}=0.1$~s. After selecting $\alpha=\mu_{i}=20$ and $p_{i}=10$, \lem{lem4} leads to $\phi^{*}=0.034$. Letting $Q=0.04$ and $\rho=0.001$ guarantees $\Psi(\|\fbm{q}(0)\|)<\phi^{*}$ and gives $\delta=0.015$. Then, $\beta_{i}=200$ guarantees $V(0)<\phi^{*}$. The injected virtual damping is selected $k_{i}=10$ to suppress the delay-induced distortions. \fig{fig17b} ilustrates that, despite actuation limits and system uncertainties, all robot end effectors converge to the same point in~$3\times 10^{4}$~s. Detailed coordination along the $x$- and $y$-axes is displayed in~\fig{fig18}.
\begin{figure}[!hbt]
\centering
\subfigure[End effector positions in task space along $x$-axis.]{
	\label{fig18a}
	\includegraphics[width=4cm,height=2.5cm]{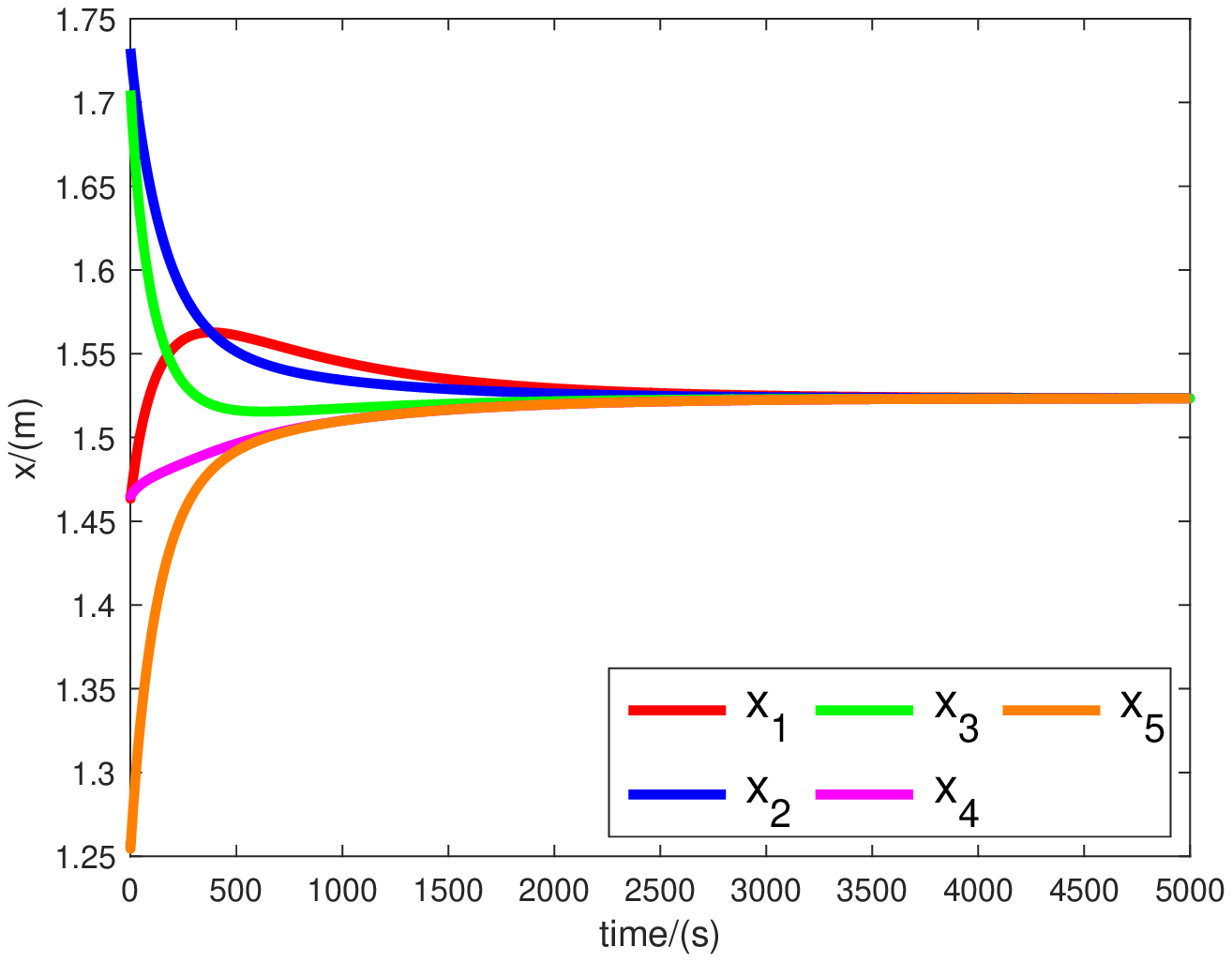}}
\subfigure[End effector positions in task space along $y$-axis.]{
	\label{fig18b}
	\includegraphics[width=4cm,height=2.5cm]{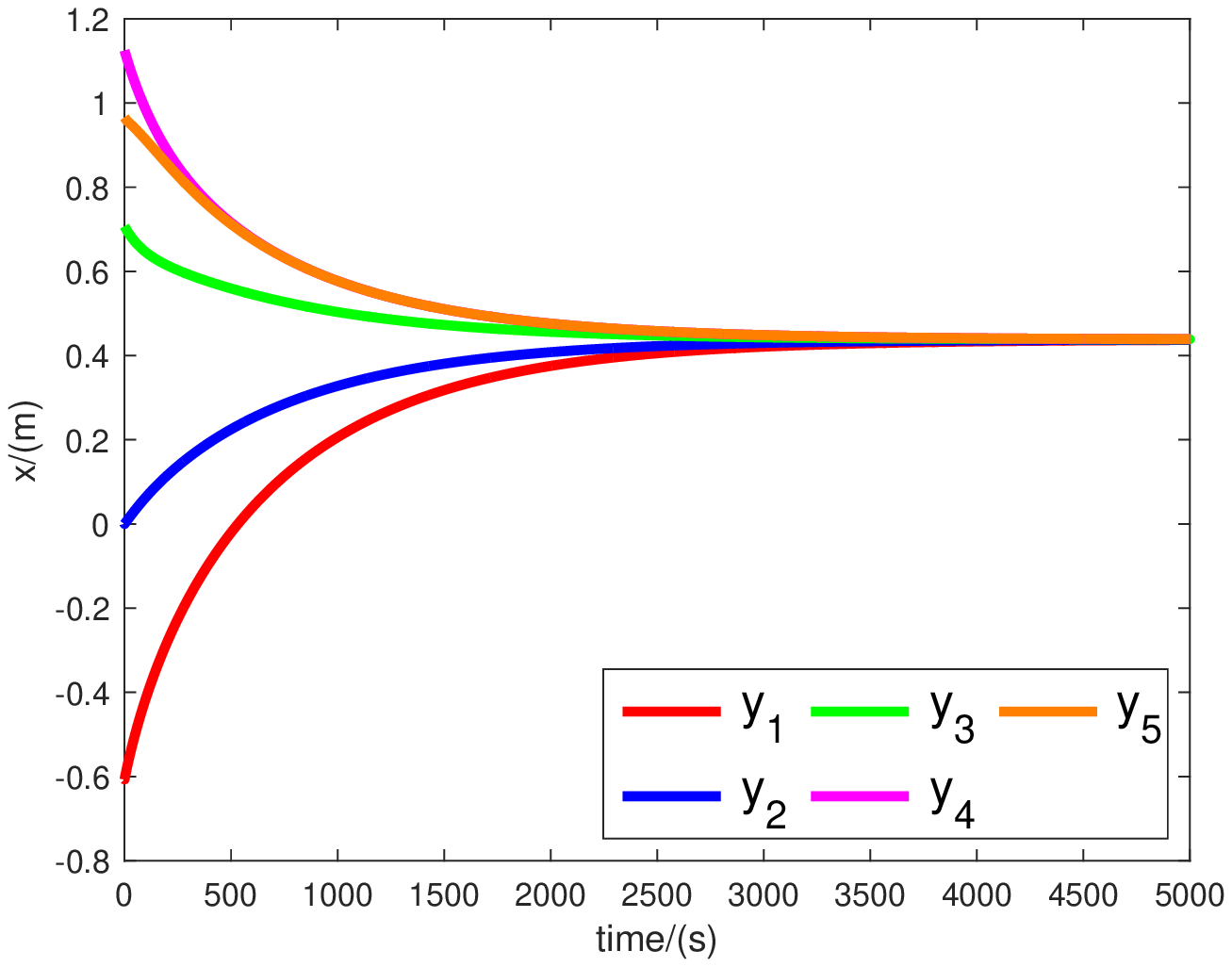}}
\caption{End effector positions of the uncertain $5$-robots MAS with bounded actuations and time-varying delays under the control~\eqref{equ32}.}
\label{fig18}
\end{figure}

\section{Conclusions}
This paper has explored the impact of actuation bounds on the synchronization with connectivity maintenance of kinematic and Euler-Lagrange multi-agent systems. Regarding actuator saturation as dynamic scaling of the control inputs has led to the conclusion that actuator saturation threatens neither the coordination nor the connectivity of kinematic multi-agent systems. Thus, conventional negative gradient-based controllers derived from generalized potential functions can achieve the connectivity-preserving consensus objective without modification. As a result, such controllers exploit the available actuation better and converge faster than controllers designed to account for the actuation bounds. For Euler-Lagrange multi-agent systems, the paper has shown that actuator saturation restricts the initial states from which synchronization with local connectivity preservation is achievable. For fully actuated Euler-Lagrange multi-agent systems, the paper has developed gradient-based controllers that can achieve consensus subject to connectivity maintenance with no velocity sensing and with parameter uncertainties, respectively. For Euler-Lagrange systems with bounded actuation, an indirect coupling control framework has decomposed the inter-agent couplings into agent-proxy couplings and inter-proxy couplings. This decomposition has led to a transformation of the actuation bounds into a bound on the potential function designed to preserve connectivity. Lastly, the framework has been extended to address the consensus of Euler-Lagrange multi-agent systems with parameter uncertainties, time-varying delays and connectivity preservation simultaneously.

\bibliography{bibi}
\end{document}